\documentclass[12pt, reqno]{amsart}
\usepackage{amsthm}
\usepackage{amsfonts}
\usepackage{comment}
\usepackage{amssymb}
\usepackage{comment}
\usepackage{mathtools}
\usepackage[usenames,dvipsnames]{xcolor}
\usepackage[backref, colorlinks=true, citecolor=Green]{hyperref}
\usepackage[alphabetic,backrefs,abbrev, lite]{amsrefs}
\usepackage{tikz-cd}

\usepackage{fullpage}
\usepackage{bm}
\usepackage{breqn}
\usepackage{enumitem}
\usepackage{setspace}
\usepackage{xcolor}
\usepackage{multirow}
\usepackage{stmaryrd}
\usepackage{caption}
\usepackage{diagbox}
\usepackage{microtype}

\numberwithin{equation}{section}
\newtheorem{Theorem}{Theorem}[section]
\newtheorem{thmx}{Theorem}

\newtheorem*{Theorem*}{Theorem}
\newtheorem{Corollary}[Theorem]{Corollary}
\newtheorem{Lemma}[Theorem]{Lemma}
\newtheorem{Proposition}[Theorem]{Proposition}
\newtheorem*{Proposition*}{Proposition}

\theoremstyle{definition}
\newtheorem{Definition}[Theorem]{Definition}
\newtheorem{Example}[Theorem]{Example}
\newtheorem{remark}[Theorem]{Remark}
\newtheorem{Question}[Theorem]{Question}

\newcommand{\calH}{\mathcal{H}}
\newcommand{\calG}{\mathcal{G}}
\newcommand{\G}{\Gamma}
\newcommand{\R}{\mathbb{R}}
\newcommand{\Z}{\mathbb{Z}}
\newcommand{\F}{\mathbb{F}}
\newcommand{\C}{\mathbb{C}}
\newcommand{\M}{\mathbb{M}}
\renewcommand{\H}{\mathbb{H}}
\DeclareMathOperator{\lcm}{lcm}
\renewcommand{\Im}{i}
\newcommand{\defi}[1]{\textbf{\textsf{#1}}}

\DeclareMathOperator{\Aut}{Aut}

\DeclareMathOperator{\SL}{SL}
\DeclareMathOperator{\Ch}{Ch}
\DeclareMathOperator{\tr}{tr}

\newcommand{\gbar}{\bar{g}}

\title{Replicable functions arising from code-lattice VOAs fixed by automorphisms}

\author{Lea Beneish}
\address{Department of Mathematics, University of California, Berkeley, Berkeley, CA 94720}
\email{leabeneish@math.berkeley.edu}
\urladdr{}

\author{Jennifer Berg}
\address{Department of Mathematics, Bucknell University, Lewisburg, PA 17837}
\email{jsb047@bucknell.edu}
\urladdr{\url{https://sites.google.com/view/jenberg}}

\author{Eva Goedhart}
\address{Department of Mathematics, Franklin \& Marshall College, Lancaster, PA 17604}
\email{eva.goedhart@fandm.edu}
\urladdr{\url{https://www.evagoedhartphd.com/}}

\author{Hussain M. Kadhem}
\address{Department of Mathematics, University of California, Berkeley, Berkeley, CA 94720}
\email{hmk@math.berkeley.edu}
\urladdr{}

\author{Allechar Serrano L\'opez}
\address{Department of Mathematics, Harvard University, Cambridge, MA 02138}
\email{serrano@math.harvard.edu}
\urladdr{\url{https://allechar.org/}}

\author{Stephanie Treneer}
\address{Department of Mathematics, Western Washington University, Bellingham, WA 98225}
\email{trenees@wwu.edu}
\urladdr{}

\begin{document}
\maketitle

\begin{abstract}

We ascertain properties of the algebraic structures in towers of codes, lattices, and vertex operator algebras (VOAs) by studying the associated subobjects fixed by lifts of code automorphisms. In the case of sublattices fixed by subgroups of code automorphisms, we identify replicable functions that occur as quotients of the associated theta functions by suitable eta products. We show that these lattice theta quotients can produce replicable functions not associated to any individual automorphisms. Moreover, we show that the structure of the fixed subcode can induce certain replicable lattice theta quotients and we provide a general code theoretic characterization of order doubling for lifts of code automorphisms to the lattice-VOA. Finally, we prove results on the decompositions of characters of fixed subVOAs.
\end{abstract}

\section{Introduction}
In this paper we investigate the connections between linear codes, lattices, and vertex operator algebras (VOAs). These three algebraic structures arose independently in mathematics, but are related by constructions which build lattices from codes, and VOAs from lattices. Moreover, the automorphism groups associated to these structures are linked to the theory of moonshine which explores the surprising connections between sporadic simple groups and modular objects. As a motivating example, we consider the following structures associated to the extended Golay code.

The extended binary Golay code $\mathcal{G}$ is a a \textit{doubly even self-dual} binary linear code of length $24$ which can be used to construct the Leech lattice $\Lambda_{24}$. In turn, the Leech lattice is involved in the construction of the monster module $V^{\natural}$, an orbifold of the Leech lattice vertex operator algebra. Crucially, the automorphism groups of each of these three objects are related to sporadic simple groups: the Matheiu group $M_{24}$, the Conway group $Co_0$ (which is equal to $2.Co_1$), and the monster group $\mathbb{M}$, respectively, each of which is a subquotient of the next. This suggests a tower involving codes, lattices, and VOAs:
\vspace{-1mm}
\begin{center}
\begin{tikzcd}
\mathbb{M}= \Aut(V^{\natural}) \arrow[d, no head] \\
Co_0=\Aut(\Lambda_{24}) \arrow[d, no head]             \\
M_{24}=\Aut(\mathcal{G}).            
\end{tikzcd}
\end{center}
A similar tower construction which does not require orbifolds consists of the extended Hamming code $\calH$, the $E_8$ lattice, and the $E_8$ lattice-VOA.

In general, one can construct a lattice from a given binary linear code, and there is a correspondence between (doubly even) self-dual binary linear codes of length $N$ and certain (even) unimodular lattices in $\mathbb{R}^N$. Given an even positive definite lattice, one can construct its corresponding lattice vertex operator algebra. To further develop the analogies between these three structures, we consider the modular forms related to each. It is well known that one can associate a theta function to any even positive definite lattice and such functions are holomorphic on the upper half plane. If the lattice is also unimodular then these theta functions are modular forms. 
For a rational $C_2$-cofinite vertex operator algebra $V$, Zhu showed that the characters of the irreducible modules of $V$ form a vector-valued modular function for $\SL_2(\mathbb{Z})$ with a multiplier system \cite{Zhu}. Moreover, the characters of fixed point subVOAs under certain automorphisms of $V$ are modular forms for congruence subgroups \cite{DLM}.

It is natural to investigate the automorphisms that fix an algebraic structure as a means of better understanding the structure itself. The central focus of this paper is the study of fixed sub-objects in general towers consisting of a code $C$, code-lattice $L$ (see \S\ref{sec: codelattices}), and lattice-VOA $V_L$ (see \S\ref{sec: latticeVOAs}), under automorphisms of $C$ suitably lifted to act on $L$ and $V_L$. We note that there is more than one way to construct a lattice from a code and to construct a VOA from a lattice, for details on the constructions we consider see sections \ref{subsec: constA}, \ref{subsec: equivconst}, and \ref{subsec: Lattice-VOA construction}. In particular, we study sublattices of $L$ fixed by automorphisms via their associated theta functions and subVOAs of $V_L$ fixed by lifted lattice automorphisms via their associated characters.

Monstrous moonshine inspired the study of replicable functions arising from sublattices of the Leech lattice fixed by individual automorphisms of the Golay code \cite{KT86} and later by the full group of lattice automorphisms \cite{L89}. Soon after, a similar analysis was done for the $E_8$ lattice \cite{CLL92}. However, one need not focus only on the action of individual automorphisms. More generally, we study sublattices fixed by \emph{subgroups} of code automorphisms embedded into the automorphism group of the lattice. In doing so, we give a natural generalization of the quotient of the lattice theta function by the eta product determined by the cycle type of an automorphism, which we hereafter refer to as the \emph{lattice theta quotient}. This leads to the following result. 

\begin{thmx} \label{thm: mainintro} Let $C$ be a doubly even self-dual binary linear code of length $N$ and let $L$ be the associated code-lattice. Let $G \subset \Aut L$ be a subgroup of automorphisms in the image of $\Aut C$ under the natural embedding and let $L^G$ be the sublattice fixed by $G$. Then the lattice theta quotient $(\theta_{L^G}(q)/\eta_G(q))^{24/N}$ is a {weakly holomorphic} modular function. In the case that $G = \langle g \rangle$ is a cyclic subgroup, this recovers the usual lattice theta quotient associated to $g$. Moreover, this construction can produce non-monstrous replicable functions not necessarily associated to any individual automorphisms.
\end{thmx}

Various orbit types of subgroups that appear in the automorphism group of the extended Hamming code $\calH$ embedded in the automorphism group of the $E_8$ lattice similarly occur for higher rank unimodular code-lattices. In many cases the cycle type of an element of $\Aut C$ (or more generally the orbit type of a subgroup of $\Aut C$) is not sufficient to determine the corresponding sublattice of the code-lattice. However, in Proposition \ref{prop: replicable_eta_quos} we show that certain replicable functions must appear as lattice theta quotients for particular pairings of subgroup orbit type with sublattice isomorphism class. 

In fact, the structure of the subcode $C^{\overline{g}}$ fixed by an automorphism $\overline{g}\in\Aut{C}$ is sometimes enough to guarantee that the lifted automorphism $g\in \Aut L$ produces a certain lattice theta function. We prove the following results for cycle types of automorphisms of $C$.

\begin{thmx} \label{thm: rA1^N/r}
    Let $C$ be a doubly even self dual linear code of length $N$, let $r | N$, and let $\gbar \in \Aut C$ have cycle type $r^{N/r}$. Suppose that the fixed subcode $C^\gbar$ has dimension $\dim C^{\gbar} = \frac{1}{r} \dim C$ spanned by $\{B_1, \dots, B_{N/2r}\}$. Further suppose $B_i \cap B_j = \emptyset$ for all $i \ne j$ and $\cup_i B_i = \{1, \dots, N\}$. Then the corresponding sublattice fixed by $g = \iota(\gbar)$ has theta series given by $\theta_{L^g}(q) = \vartheta_3(q^r)^{N/r}$. In particular the theta series of the fixed sublattice is the same as that of the lattice $A_1(r)^{N/r}$.
\end{thmx} 

It then immediately follows from Proposition \ref{prop: replicable_eta_quos} that the lattice theta quotients corresponding to the sublattices in Theorem~\ref{thm: rA1^N/r} are replicable. In the case of cycle type $2^{N/2}$, Proposition~\ref{prop:2DN/2*} gives another characterization of a fixed subcode producing a lattice theta function of $D_{N/2}^\ast$ type with non-replicable lattice theta quotient. In Theorem~\ref{thm: orddoubcode} we give a general code theoretic characterization describing when even order lattice automorphisms $g$ lift to automorphisms $\hat{g}$ of the lattice-VOA $V_L$ that have order doubling. We use this result to show that the $2^{N/2}$ replicable cycle type exhibits order doubling when lifted to a VOA automorphism, while the non-replicable $2^{N/2}$ cycle type does not. Interestingly, these two cycle types are entwined via the characters of their fixed point subVOAs. In particular, the associated lattice theta quotients satisfy the following identities.

\begin{thmx} \label{thm: etaquoandcharidentities} 
Suppose there exists an automorphism $g_1 \in \Aut L$ with cycle type $2^{N/2}$ and fixed sublattice isometric to $A_1(2)^{N/2}$. Let $L_0$ denote the kernel for order doubling associated to the lift of $g_1$ to the VOA automorphism $\hat{g_1} \in \Aut V_L$ and let $V_L^{+}$ denote the subVOA fixed by the automorphism $-id \in \Aut L$
\begin{equation}
    \frac{\theta_{L^{g_1}}(q)}{\eta(q^2)^{N/2}} = \Ch V_L^{\hat{g_1}}(q) - \Ch V_{L_0}^{+}(q) + \Ch V_{D_{N/2}}(q^2).   
\end{equation}
Moreover, if there exists an automorphism $g_2 \in \Aut L$ with cycle type $2^{N/2}$ and fixed sublattice isometric to $D_{N/2}^\ast(2)$, then the characters of the fixed point subVOAs are related via
 \begin{equation} 
     \frac{\theta_{L^{g_2}}(q)}{\eta(q^2)^{N/2}}  = 2(\Ch V_L^{\hat{g_2}}(q) - \Ch V_L^{+}(q) ) - \bigl(\Ch V_L^{\hat{g_1}}(q) - \Ch V_{L_0}^{+}(q)\bigr) + \Ch V_{D_{N/2}}(q^2).    
\end{equation}  
\end{thmx} 

Finally, just as we considered sublattices fixed by noncyclic subgroups lifted from the automorphism group of the code, we consider characters of subVOAs fixed by noncyclic subgroups. In the more straightforward setting in which no order doubling of elements occurs, fixed subVOA characters may decompose in a way that reflects the structure of the group.  For example, we prove the following result pertaining to semidirect products of order $pq$.

\begin{thmx} \label{thm: charpq} Let $p$ and $q$ be primes such that $q>p$ and $q\equiv 1 \pmod{p}$ and let $\Z_q \rtimes \Z_p $ be a subgroup of the automorphism group of an even positive definite lattice $L$. The characters of the fixed point subVOAs of the lattice-VOA $V_L$ satisfy the following relation
\[ p\Ch V^{\Z_q \rtimes \Z_p}= \Ch V^{\Z_q} +  p\Ch V^{\Z_p} - \Ch V_{L}. \]
\end{thmx}

\subsection{Outline} In Section~\ref{sec: preliminaries}, we provide further background on the main objects in this paper, and in particular define codes, lattices, and vertex operator algebras. We describe several equivalent constructions of a lattice from a code $C$ in Section~\ref{sec: codelattices} and make explicit the action of the automorphism group of the code on elements of the related code-lattice. This allows us to study theta quotients for sublattices fixed by automorphisms of the code. We give a new definition of a lattice theta quotient associated to a sublattice fixed by a \emph{subgroup} of $\Aut C$ and give a proof of Theorem~\ref{thm: mainintro} together with examples. In Section~\ref{sec: latticeVOAs}, we review the lattice-VOA construction and discuss automorphisms of VOAs that are lifted from the underlying lattice. We recall the definition of characters for subVOAs fixed by finite cyclic groups of autormorphisms in Section~\ref{sec: charsfixedsubVOAS}.

In Section \ref{sec: repfunctionscodelats} we consider replicable functions that can be common to code-lattices of arbitrary rank by considering groups of automorphisms whose orbit types correspond to partitions of 8. We characterize theta quotients that are replicable functions based on their fixed sublattices in Proposition~\ref{prop: replicable_eta_quos} and catalog the corresponding data for the 12 even unimodular code-lattices of rank at most $24$. Moreover, we show that by considering subgroups of automorphisms, we recover additional non-monstrous replicable theta quotients not associated to any cyclic subgroups. 

In Section~\ref{sec: characterizationcodes} we give  characterizations of certain fixed sublattices of a code-lattice $L$ by automorphisms in the image of the embedding from $\Aut C$ in terms of properties of the fixed subcodes of $C$. In Section~\ref{sec: coeffsVOAchars} we prove identities regarding coefficients of characters of fixed subVOAs. Finally, in Section~\ref{sec: decompVOAchars} we consider how the structure of subgroups of $\Aut C$ is reflected in the structure of their corresponding fixed subVOAs. We prove results which relate characters of subVOAs fixed by semidirect products with the characters of the subVOAs fixed by the component subgroups in Theorem \ref{thm: charpq} and Theorem \ref{thm: charp2q}. We end the paper in Section~\ref{sec: questions} with questions for further study.

\subsection{Acknowledgements}
Our research group was formed during the Rethinking Number Theory II virtual workshop in July 2021. Rethinking Number Theory, now an American Institute of Mathematics (AIM) Research Community, promotes accessible and joyful collaborative research experiences in number theory. We also wish to thank AIM for their support of our group through an ongoing SQuaRE.We are grateful to Richard Borcherds, Robert Griess, Christoph Keller, Jeffrey Lagarias, and Geoffrey Mason for helpful conversations and email correspondences. We would further like to thank John Duncan and Christoph Keller for thoughtful comments on an earlier draft. This material is based upon work supported by the National Science Foundation under Grant No. DMS-1928930 while J.B.\ was in residence at the Simons Laufer Mathematical Sciences Institute in Berkeley, California, during the Spring 2023 semester. Finally, we are grateful for our faithful research companion, Otter Kadhem.

\section{Preliminaries} \label{sec: preliminaries}

Here we review basic properties of the objects central to our study: linear codes, lattices, and vertex operator algebras. We also introduce modular forms and replicable functions. 

\subsection{Linear Codes} \label{subsec: linearcodes}

We will briefly introduce some necessary properties of linear codes. For a thorough introduction, see for example \cites{Hill,Ebeling}. Let $\F_q$ be the finite field of $q$ elements. A \defi{linear code} $C$ is any $k$-dimensional subspace of $\F_q^n$. The weight of a codeword $c \in C$ is the number of its nonzero entries. The distance $d$ of $C$ is the minimal weight of its nonzero codewords, or equivalently the minimal number of positions in which any two codewords differ. We will call such $C$ an $[n,k,d]_q$-code. In this paper, all codes will be binary linear codes, so henceforth $q=2$ and we will omit this subscript from the code parameters. An $[n,k,d]$-code $C$ can be specified by giving a generator matrix $G \in \F_2^{k\times n}$ whose rows form a basis of $C$.

A code is \defi{doubly even} if each codeword has weight divisible by 4. It is \defi{self-dual} if it is equal to its dual code
\[C^\perp =\{x\in \F_2^n\mid \textstyle \sum_{i=1}^n x_i c_i =0\mbox{ for all $c\in C$}\}.\]A code $C$ of length $n$ and its dual always satisfy the dimension formula $\dim C + \dim C^\perp = n$, hence a self dual code has dimension $\frac{n}{2}$.

There is a natural action of the symmetric group $S_n$ on $\F_2^n$ by means of coordinate permutation. The subgroup of linear isomorphisms of $\F_2^n$ arising from $\sigma\in S_n$ that fix the code $C$ forms the automorphism group $\Aut C$, that is, \[ \Aut C = \{ \sigma \in S_n : \sigma(C) = C\}.\] The \defi{cycle type} of $\gbar\in \Aut C$ describes the disjoint cycle decomposition of $\gbar$, written in the form $\prod_t t^{r_t}= 1^{r_1}2^{r_2}\cdots n^{r_n}$. We will denote the subspace of $C$ fixed by an automorphism $\gbar\in \Aut C$ as $C^\gbar$. 

\begin{Example} The extended Hamming code $\calH$ is the doubly even self-dual $[8,4,4]$-linear code with generator matrix
\[ \begin{bmatrix}
1 & 0 & 0 & 0 & 0 & 1 & 1 & 1 \\
0 & 1 & 0 & 0 & 1 & 0 & 1 & 1 \\
0 & 0 & 1 & 0 & 1 & 1 & 0 & 1 \\
0 & 0 & 0 & 1 & 1 & 1 & 1 & 0 \\
\end{bmatrix}. \]
\end{Example}

The automorphism group $\Aut \calH$ has order $1344$. The element $(2,8,4,6)(3,5)\in \Aut \calH \subset S_8$ has cycle type $1^2 2^1 4^1$.

\subsection{Lattices}

Here we establish some basic properties of lattices. For more details, see, for example \cite{Ebeling}. A \defi{lattice} in $\mathbb{R}^n$ is a discrete subset $L \coloneqq \mathbb{Z}e_1 \oplus \dots \oplus \mathbb{Z}e_m$ such that the set $\{e_1,e_2,\dots,e_m\}$ is linearly independent in $\mathbb{R}^n$. We consider only \defi{positive definite} lattices, for which the norm of every $\lambda\in L$ with respect to the standard Euclidean inner product $\langle -,-\rangle$ on $\mathbb{R}^n$ is positive. With this notation, $L$ is an $n$-dimensional lattice of rank $m$. The determinant of $L \subset \mathbb{R}^n$ is the determinant of the inner product matrix for a basis of $L$. The dual lattice of $L$ is $L^\ast \coloneqq \{x\in \mathbb{R}^n\mid \langle x,y\rangle\in \Z\mbox{ for all }y\in L\}$. An automorphism of $L$ is a length-preserving linear endomorphism of $L$. We let $\Aut L$ denote the group of automorphisms of $L$.

Next we define some lattice properties which are desirable for lattice-VOA constructions. A lattice $L$ is \defi{even} if the Euclidean norm $\langle \lambda ,\lambda \rangle \in 2\mathbb{Z}$ for every $\lambda\in L$.  A lattice is called \defi{unimodular} if $\det(L) = \pm1$. The unique even, positive definite, unimodular lattice of rank 8 is called $E_8$, so named for its connection to the complex simple Lie algebra of the same name. This lattice can be constructed from the extended Hamming code $\mathcal{H}$, as described in Section~\ref{subsec: constA}, and will serve as the setting for examples throughout the paper. The Leech lattice $\Lambda_{24}$ is an even, positive definite, unimodular lattice of rank 24 which has connections to Conway-Norton moonshine.

A natural counting function associated with a lattice $L$ is the \defi{lattice theta function} $\theta_L(z)$, whose coefficients record the number of lattice points of fixed norm. Let 
\begin{equation} \label{eqn: latticethetafunction} \theta_L(z):=\sum_{\lambda\in L} q^{\langle \lambda ,\lambda \rangle/2},\end{equation}
where $q:=e^{2\pi iz}$. For example, $\theta_\Z(z)=\sum_{n\in\Z}q^{n^2/2}$ is the theta function for $\Z\subset\R$. Given a lattice $L$ and positive integer $n$, we let $L(n)$ denote the lattice $L$ with inner product scaled by $n$. Then $\theta_{L(n)}(q) = \theta_L(q^n)$.

\subsection{Modular forms}

The modular group $\SL_2(\Z)$ acts on the complex upper half plane $\H$ by linear fractional transformations $\gamma z=\frac{az+b}{cz+d}$ for $\gamma=(\begin{smallmatrix}a&b\\c&d\end{smallmatrix})$ and $z\in\H$. The elements of $\SL_2(\Z)$ are generated by $S=(\begin{smallmatrix}0&-1\\1&0\end{smallmatrix})$ and $T=(\begin{smallmatrix}1&1\\0&1\end{smallmatrix})$. Modular forms are functions on $\mathbb{H}$ that transform nicely with respect to $\SL_2(\Z)$ or certain congruence subgroups $\G\subseteq \SL_2(\Z)$. 

A \defi{weakly holomorphic modular form} of weight $k\in\frac12\Z$ for $\Gamma\subseteq \SL_2(\Z)$ is a function $f:\H\to\C$ which is holomorphic on $\H$, meromorphic when extended to the cusps of the Riemann surface $\H/\Gamma$, and satisfies
\begin{equation*}
    f(\gamma z)=\varepsilon_\gamma(cz+d)^kf(z)
\end{equation*}
for all $\gamma=(\begin{smallmatrix}a&b\\c&d\end{smallmatrix})\in\G$ and $z\in\H$, where $\varepsilon_\gamma$ is a particular root of unity depending on $\gamma$. If, in addition, $f$ is holomorphic at the cusps of $\H/\Gamma$, then $f$ is a \defi{holomorphic modular form}. The space of weakly holomorphic (resp. holomorphic) modular forms of weight $k$ for $\G$ is given by $\mathcal{M}_k(\G)$ (resp. $M_k(\G)$). Modular forms of weight zero are called \defi{modular functions}.

Dedekind's eta function is given by \begin{equation*}\eta(z):=\eta(q)=q^{1/24}\prod_{n=1}^\infty (1-q^n)\end{equation*} 
for $q=e^{2\pi i z}$. 
We will consider \defi{eta quotients}, which generally have the form $\prod_{d\mid N}\eta(dz)^{r_d}$ for some integer $N\geq 1$ and $r_d\in\Z$. Eta quotients are holomorphic on $\H$ and meromorphic at the cusps. Thus under certain conditions which guarantee modularity \cite{Ono04}*{Theorem 1.64}, eta quotients are weakly holomorphic modular forms. We refer the reader to \cites{CS,Ono04} for more detailed treatments of modular forms.

\subsection{Replicable functions} 
The following formulation of replicable functions arises from McKay and Sebbar \cite{MS07}. We first define {Faber polynomials}; see also \cite{ACMS92} for more on replicable functions and another formulation of Faber polynomials.

\begin{Definition}[Faber Polynomial]\label{def:FaberPolynomial}
Let $f(q)$ be a power series, $f(q)=\frac{1}{q}+\sum_{n=1}^\infty a_n q^n$ for $q=e^{2\pi i z}$ with $z\in\H$, $a_n\in\C$. For each $k\in\Z^+$, there exists a unique monic polynomial $F_k$, depending on the coefficients of $f$, such that
\[F_k(f(q))=\frac{1}{q^k}+O(q) \text{ as } q\to \infty. \]
The $q$-expansion of the Faber polynomial of $f(q)$, denoted $F_k(f(q))$ has the form \[F_k(f(q))=\frac{1}{q^k}+k\sum_{n=1}^\infty a_{n,k}q^n,\] where $a_{n,1}=a_n$ and the double sequence $a_{n,k}$ is symmetric, i.e., $a_{n,k}=a_{k,n}$.
\end{Definition}
 We are now ready to state what it means for a formal power series $f(q)$ as above to be a replicable function in terms of the coefficients of its Faber polynomial. In particular, for $f(q)$, we denote the Faber polynomial of degree $k$ by $F^{f}_k(z)$ to indicate its dependence on $f$.  One can show that $F^f_0(z) = 1$, $F^f_1(z) = z$, $F^f_2(z) = z^2 - 2a_1$, $F^f_3(z) = z^3 - 3a_1z - 3a_2$ and that the Faber polynomals satisfy the following recurrence relation  \[F^{f}_{k+1}(z)=zF^{f}_k(z)-\sum_{n=1}^{k-1}a_{k-n}F^{f}_n(z)-(k+1)a_k.\] 

\begin{Definition}\label{def:ReplicableFunction}
The function $f(q)$ is said to be \defi{replicable} if $a_{n,k}=a_{r,s}$ whenever $\gcd(n,k)=\gcd(r,s)$ and $\lcm(n,k)=\lcm(r,s)$. 
\end{Definition}
Some of the simplest examples of formal power series are the functions $f_c(q)=q^{-1}+cq$ for some $c\in\R$. It is straightforward to show that each $f_c$ is a replicable function. Another well-known class of replicable functions is related to Monstrous Moonshine. The Monstrous Moonshine Conjecture~\cite{CNMoonshine} asserts that there exists an infinite dimensional graded Monster module, $V=\bigoplus V_n$, such that for the conjugacy class of $g\in\M$, denoted $[g]$, the modular functions $T_{g}(\tau)$ with $\tau\in\H$, associated to $[g]$ are sums of traces of irreducible representations of the Monster, i.e., $T_{g}(\tau)=\sum \tr (g\vert V_n)q^n$. We refer to these $T_{g}(\tau)$ as ``monstrous moonshine functions,'' also known as McKay-Thompson series. These  modular functions are \emph{hauptmoduln}, that is, each is a generator for the field of modular functions for a certain subgroup of $\SL_{2}(\R)$ depending on $g$. Further, each monstrous moonshine function $T_{g}(\tau)$ is a replicable function~\cite{G06}. 

However, the converse is not true, not every replicable function is a hauptmodul let alone appears as a monstrous moonshine function. As noted in~\cite{G06}, Norton conjectured that any replicable function with rational coefficients is either a hauptmodul, or a `modular fiction,' that is, a function $f_0(\tau)=q^{-1}$ or $f_{\pm 1}(\tau)=q^{-1}\pm q$. 

\subsection{Vertex Operator Algebras}
We begin by establishing the notation we will use throughout the paper. We recall some basic definitions and properties of vertex operator algebras (VOAs) and their twisted modules and refer the reader to \cites{FBZ,FLM88,LL12} for explicit details.
\begin{Definition} A \defi{vertex operator algebra} is a complex vector space $V$ equipped with two distinguished vectors, the {vacuum element} $\bf{1}$ and the {conformal vector} $\omega$. There is a map on $V$ called a {vertex operator} denoted $Y(\cdot, z): V \to \text{End}(V)\llbracket z, z^{-1}\rrbracket$ which assigns to each vector $v\in V$ a formal power series 
$$Y(v, z) \coloneqq \sum_{n\in\mathbb{Z}} v(n) z^{-n-1}.$$ 
The tuple $(V,{\bf{1}},\omega, Y)$ must satisfy several axioms stated in \cite{FLM88}*{\S8.10}. In particular, if $Y(\omega, z)\coloneqq \sum_{n\in\mathbb{Z}} L(n) z^{-n-2}$, with $L(n)$ defined as the coefficients $\omega(n+1)$, then for any $n_1,n_2\in\Z$, we have $$[L(n_1),L(n_2)]=(n_1-n_2)L(n_1+n_2)+ \textstyle \frac{1}{12}(n_1^3-n_1)\delta_{n_1+n_2,0}\,c,$$ where $\delta_{n_1+n_2,0}$ is the Kronecker delta function and $[-,-]$ is a Lie bracket. We refer to $c$ as the \defi{central charge} of $V$. That is, the coefficients of the vertex operator attached to the conformal vector $\omega$ generate a copy of the Virasoro algebra of central charge $c$.

VOAs admit a $\mathbb{Z}$-grading (bounded from below) so that $V=\bigoplus_{n\in \mathbb{Z}} V_n$. This grading on $V$ comes from the eigenspaces of the $L(0)$ operator. That is, $V_n:=\{ v\in V \mid L(0)v=nv \}$. The smallest $n$ for which $V_n\neq 0$ is called the {conformal weight} of $V$ and is denoted $\rho(V)$. We say $V$ is of \defi{CFT-type} if $\rho(V)=0$ and $V_0=\mathbb{C}{\bf{1}}$.
\end{Definition}

\begin{Definition}
A \defi{$V$-module} is a vector space $W$ equipped with an operation 
\[Y_W (\cdot, z) \colon V \to \text{End}(W) \llbracket z, z^{-1}\rrbracket\] which assigns to each $v\in V$ a formal power series $Y_W(v, z):=\sum_{n\in\mathbb{Z}} v^{W}(n) z^{-n-1}$. Again, $(W,{\bf{1}},\omega, Y_W)$ is subject to several axioms that can be found in \cite{FBZ}*{\S5.1}. If the only submodules of $W$ are $0$ and $W$ itself, then $W$ is called {simple} or \defi{irreducible}. 
\end{Definition}

We say $V$ is a \defi{rational} VOA if every admissible
$V$-module decomposes into a direct sum of (ordinary) irreducible modules and $V$ is \defi{holomorphic} if it is rational and has a unique irreducible module which must be $V$ itself. Given a $V$-module $W$ with a grading, one can define a $V$-module, $W'$, that is the graded dual space of $W$ as a vector space. (For a definition of the dual module, refer to \cite{FHL93}*{\S 5.2}). We say a vertex algebra $V$ is \defi{self-dual} if the module $V$ is isomorphic to its dual $V'$ as a $V$-module. In~\cite{Zhu}, Zhu introduced a finiteness condition on VOAs. We say $V$ is {$C_2$-cofinite} if $C_2(V ) := \text{span}\{v(2)w \mid v, w \in V \}$ has finite codimension in $V$. 
A VOA is called \defi{strongly rational} if it is rational, $C_2$-cofinite,  self-dual, and of CFT-type.

In analogy to the lattice theta series, the graded characters of $V$ serve as counting functions. Moreover, we have that graded characters of subVOAs of rational, $C_2$-cofinite VOAs $V$ exhibit nice transformation properties as functions of the upper half plane by work of Zhu \cite{Zhu}, Dong–Li–Mason \cite{DLM}, and Dong–Lin–Ng \cite{DLNg}.

\section{Code-lattices and their automorphisms} \label{sec: codelattices}

In this section we introduce code-lattices which are lattices constructed from binary linear codes via Construction A; for other constructions of lattices from codes, see \cite{CS99}*{Chapters 5,7}. We consider their automorphism groups and investigate, in particular,  properties of sublattices fixed by subgroups of automorphisms of the code, which can be viewed as automorphisms of the lattice. While there are several equivalent constructions of such lattices, they differ from one another from a computational perspective. To determine sublattices fixed by automorphisms, one approach requires an embedding of the automorphism group of the code into the automorphism group of the corresponding lattice. Such an embedding becomes increasingly computationally expensive to implement as the rank of the lattice grows. Recall, for example, that $\Aut E_8$ is the Weyl group of type $E_8$ with order $696729600$ and the automorphism group of the Leech lattice is the Conway group $Co_0$ with order $\approx 8.3\times10^{18}$. 

Thus an approach that allows for computations purely via the automorphism group of the code is useful since although these groups also tend to grow in size as the length of the code grows, they are relatively small in comparison. For example, the automorphism group of the Hamming code $\calH$ has order $1344$ and the automorphism group of the Golay code $\mathcal{G}$ is the Matheiu group $M_{24}$ which has order $244823040$. We therefore provide details for two constructions of such lattices, their automorphisms, and fixed sublattices and give concrete examples whenever possible. With these tools in hand, we will then be able to determine modular functions associated to various fixed sublattices in Section \ref{sec: repfunctionscodelats} as well as modular functions associated to fixed point lattice vertex operator algebras in Section \ref{sec: coeffsVOAchars}.

\subsection{Code-lattices and their automorphisms via Construction A} \label{subsec: constA}
Let $C \subset \F_2^N$ be a binary linear code of length $N$ and consider the $\mathbb{Q}$-vector space $U$ spanned by $\{ \beta_1, \dots, \beta_N \}$ with bilinear form $\langle -, - \rangle \colon U\times U \to \mathbb{Q}$ defined by $\langle \beta_i, \beta_j \rangle = \frac{1}{2} \delta_{i,j}$. Then the lattice $L \subset U$ defined by
\[ L \coloneqq L(C) \coloneqq \left\{ \sum_i a_i \beta_i : a_i \in \Z , (a_1 + 2\Z, \dots, a_N + 2\Z) \in C \right\} \]
is called the \defi{code-lattice} of $C$. This construction is often referred to as \defi{Construction A} \cite{CS99}*{\S 7.2}. 
The code-lattice $L$ is even if and only if $C$ is a doubly even code, and is unimodular if and only if $C$ is self dual. We will assume from now on that all code-lattices $L$ will be even and unimodular. In addition, we will require that $L$ is positive definite, to guarantee that the group of automorphisms is finite.

Now consider the natural embedding of the symmetric group 
\begin{align} \label{eqn: iota} \iota \colon S_N \hookrightarrow \Aut(\Z^N),\end{align}
where each $\sigma\in S_N$ permutes the basis vectors, in particular $\sigma(\beta_j) = \beta_{\sigma(j)}$ for $1 \le j \le N$. When $L$ is the code-lattice of a code $C$ of length $N$, the group $\iota(\Aut C)$ is a subgroup of $\Aut L$ via this embedding. For any $g \coloneqq \iota(\gbar) \in\Aut L$, we may freely consider its cycle type to be the cycle type of the associated $\gbar\in \Aut C$. The cycle type of an arbitrary lattice automorphism, not necessarily in the image of $\iota$, may be computed more generally \cite{GK19}*{Appendix A}. This recovers the usual cycle type in the case that $g = \iota(\gbar)$.

\begin{remark} We have alluded to the relationship between the Leech lattice $\Lambda_{24}$ and the extended binary Golay code $\calG$. However, Construction A applied to $\calG$ produces the Niemeier lattice $N(A_1^{24})$ which is not isometric to $\Lambda_{24}$. The construction of the Leech lattice from $\calG$ can be found in \cite{CS99}*{\S4.11}. Instead, we shall see that a related construction outlined in Section~\ref{subsec: supercodes} allows us to directly compute the action of the automorphism group $\Aut \calG$ on elements of $\Lambda_{24}$. 
\end{remark}

\subsection{An equivalent construction} \label{subsec: equivconst}
Following Tasaka \cite{Tas81}, we let $N$ be a positive integer divisible by $8$ and let $C$ be an even self-dual code of length $N$. Consider any orthogonal subset $\{\alpha_1, \dots, \alpha_N\}$ of vectors of norm $2$ in Euclidean space $\R^N$. That is, under the usual inner product, $\langle \alpha_i, \alpha_i \rangle = 2$ and $\langle \alpha_i,\alpha_j \rangle= 0$ if $i \ne j$. Form the lattice $L'$ with basis $\{\alpha_1, \dots, \alpha_N \}$. For any subset $B \subseteq \Omega \coloneqq \{1, \dots, N \}$, define $\alpha_B \coloneqq \sum_{i \in B} \alpha_i$. We will associate to a codeword $(b_1,b_2,...,b_N)\in C$ the subset $B\subseteq \Omega$ consisting of indices $i$ with $b_i\ne 0$. For example, the codeword $(1,0,1,1,0,1,0,0) \in \calH$ is associated to the set $B = \{1,3,4,6\}$. By a small abuse of notation, we will then write $B\in C$ to denote such a set associated to a codeword. In this setting, by \cite{Tas81}*{\S 1}, the usual operation of addition of codewords $B,B' \in C$ is defined by their symmetric difference as sets, i.e., \begin{equation} \label{eqn: codewordaddition} B + B' \coloneqq B \cup B' \smallsetminus B \cap B'. \end{equation}

\begin{Proposition}[\cite{Tas81}*{Proposition 1}] \label{prop:tasaka} The lattice associated to the even self-dual code $C$ defined by 
\[ L = \left \langle \alpha_i, \textstyle \frac{1}{2}\alpha_B :  1 \le i \le N,\, B \in C \right \rangle = \bigcup_{B \in C} \left\{ \textstyle \frac{1}{2} \alpha_B + L' \right\}, \]
is even and unimodular.
\end{Proposition}

The construction of Proposition~\ref{prop:tasaka} applied to the Hamming code $\calH$ produces the $E_8$ lattice. For a direct proof that the lattice constructed is isometric to the $E_8$ lattice, see \cite{Griess11}*{p.\ 63}. The virtue of this construction in general is that it allows us to explicitly describe the action of $g = \iota(\gbar)$ with $\gbar \in \Aut C$ on elements $\lambda = \frac{1}{2} \alpha_B + \sum_{i=1}^N x_i \alpha_i \in L$ in a straightforward manner. In particular, 
\begin{equation} \label{eqn:gaction} g(\lambda) \coloneqq\frac{1}{2} \alpha_{\gbar(B)} + \sum_{i=1}^N x_i \alpha_{\gbar(i)}.
\end{equation}
To determine whether $\lambda$ is invariant under $g$, consider the disjoint cycle form of $\gbar \in \Aut C$ viewed as an element of $S_N$. Such an element $\lambda$ is invariant under $g$ if and only if the set $B$ is $\gbar$-invariant and if the following condition is satisfied: when $i$ and $j$ both appear in the same disjoint cycle in the decomposition of $\gbar$, then $x_i = x_j$. We define the sublattice fixed by $g$ to be $L^g \coloneqq \{\lambda\in L\mid g(\lambda)=\lambda\}$. That is,
\begin{equation}\label{fixedsublattice}
L^g=\bigcup_{B\in C^\gbar}\left\{\textstyle \frac{1}{2} \alpha_B + \sum_i x_i \alpha_i:x_i=x_j
\mbox{ if $i$ and $j$ are in the same cycle of $\gbar$}\right\}.\end{equation}

\subsubsection{The Leech lattice and super codes} \label{subsec: supercodes} The construction in Proposition~\ref{prop:tasaka} applied to the Golay code $\calG$ again produces the Niemeier lattice $N(A_1^{24})$, so in order to obtain the Leech lattice we must use a more general construction. In full generality, the following can be applied to any even self-dual code with minimum weight strictly greater than $4$, also called \emph{super codes} in \cite{Tas81}. Such codes only exist when the length $N$ is at least $24$, and the Golay code is the unique such code (up to isomorphism) with length $24$. Keeping the notation above, we let
$\Lambda_0 \coloneqq \{ \sum_{i=1}^N x_i \alpha_i : \sum_i x_i \equiv 0 \bmod{2} \}$, $\Lambda_1 \coloneqq \{ \sum_{i=1}^N y_i \alpha_i : \sum_i y_i \equiv 1 \bmod{2} \}$, and define the lattices
\begin{eqnarray}
    \mathcal{L}_0(C) &\coloneqq& \bigcup_{B \in C} ( \textstyle \frac{1}{2}\alpha_B + \Lambda_0) \cup (\frac{1}{4} \alpha_{\Omega} + \frac{1}{2} \alpha_B + \Lambda_0) \\
    \label{eqn: altconstruction} \mathcal{L}_1(C) &\coloneqq& \bigcup_{B \in C} ( \textstyle \frac{1}{2}\alpha_B + \Lambda_0) \cup (\frac{1}{4} \alpha_{\Omega} + \frac{1}{2} \alpha_B + \Lambda_1).
\end{eqnarray}
    Both are even lattices and $\mathcal{L}_j(C)$ is unimodular precisely when $\frac{1}{8}N\equiv j \bmod {2}$ \cite{Tas81}*{\S3}. Thus, the Leech lattice is isometric to the lattice $\mathcal{L}_1(\calG)$. Given $g$ arising from $\gbar \in \Aut \calG$, the action on the elements of $\Lambda_{24} = \mathcal{L}_1(\calG)$ is 
    \begin{eqnarray} g( \textstyle \frac{1}{2}\alpha_B + \sum_i x_i \alpha_i) &= &\textstyle \frac{1}{2} \alpha_{\gbar(B)} + \sum_i x_i \alpha_{\gbar(i)}  \\
    \label{eqn: gactionsupercode} g(\textstyle \frac{1}{4} \alpha_{\Omega} + \frac{1}{2}\alpha_B + \sum_i y_i \alpha_i) &= & \textstyle \frac{1}{4} \alpha_{\Omega} + \frac{1}{2} \alpha_{\gbar(B)} + \sum_i y_i \alpha_{\gbar(i)} 
    \end{eqnarray}
    and this description will allow us to compute lattice theta series of fixed sublattices.

\subsection{Theta functions and theta quotients of fixed sublattices} \label{subsec: thetaseries} Let $C$ be a code of length $N$ and let $L$ be the associated code-lattice via either Construction $A$ or Propostion~\ref{prop:tasaka}. For a given $g = \iota(\gbar) \in \Aut L$, the theta series $\theta_{L^g}(z)$ of the $g$-fixed sublattice of $L$ is a {holomorphic} modular form of weight $k \coloneqq \frac{1}{2} \sum_t r_t$. The eta product associated to $g$, denoted $\eta_g(q)$, is of weight $k$ and defined in terms of the cycle type of $g$ by
\begin{equation}
    \eta_g(q):=\prod_t\eta(q^t)^{r_t}.
\end{equation}
Then the \defi{lattice theta quotient} $(\theta_{L^g}(q)/\eta_g(q))^{24/N}$ is a weakly holomorphic modular function. In Section~\ref{subsec: sublatsfixedbysubgroups}, we will prove Theorem~\ref{thm: mainintro} which is a generalization of this fact to theta series associated to subgroups of automorphisms $G \subset \Aut L$.

\subsubsection{Computing theta functions via Jacobi theta functions} Alternatively, given a code-lattice $L$ constructed from a code $C$ via Proposition~\ref{prop:tasaka}, we can also compute the theta function of a fixed sublattice $L^g$ by $g \in \iota(\Aut C) \subset \Aut L$ by explicitly expressing $\theta_{L^g}(q)$ in terms of the classical Jacobi theta functions
\begin{eqnarray} \label{jacobithetas}
   \vartheta_2(q)&=&\theta_{\left(\Z+ \textstyle \frac12\right)}(q)=  \sum_{n \in \Z} q^{(n + \frac 12)^2/2}\\
    \vartheta_3(q)&=&\theta_{\Z}(q)= \sum_{n \in \Z} q^{n^2/2} \\
    \vartheta_4(q) &= &\sum_{n \in \Z} (-1)^n q^{n^2/2},
\end{eqnarray}
as detailed in \cite{KT86}. Specifically, the description of the fixed sublattice given in equation \eqref{eqn:gaction} gives a decomposition of $L^g$ into a union of discrete subsets in $\R^N$, each of which can be further decomposed into a sum of discrete subsets which are contained in mutually orthogonal linear subspaces in $\R^N$. For discrete subsets $X,Y \subset \R^N$, we have $\theta(X \cup Y,z) = \theta(X,z) + \theta(Y,z)$ and when $X,Y$ are contained in orthogonal subspaces, $\theta(X + Y,z) = \theta(X,z)\theta(Y,z)$. We may use these facts in order to compute the desired lattice theta series $\theta_{L^g}(q)$. We illustrate this via an example. 

\begin{Example} \label{ex:g_invariants} Let $\gbar = (2,8,4,6)(3,5) \in \Aut \calH$ and $g=\iota(\gbar)\in \Aut E_8$. The codewords fixed by $\gbar$ are $B \in \{ \emptyset, \Omega, \{2,4,6,8\}$, $\{1,3,5,7\} \, \}$. 
Taking $B = \{2,4,6,8\}$, we find any element of $E_8$ of the form $\frac{1}{2} 
\alpha_B + \sum_i x_i \alpha_i$ invariant under $g$ must have $x_2 = x_4 = x_6 = x_8$ and $x_3 = x_5$. Thus the corresponding term in the union \eqref{fixedsublattice} is 
\[ \textstyle \frac12 \alpha_{\{2,4,6,8\}} +\Z \alpha_{\{2,4,6,8\}}+  \Z \alpha_{\{3,5\}} + \Z \alpha_{1} + \Z \alpha_{7} = \left(\Z + \textstyle \frac{1}{2}\right) \alpha_{\{2,4,6,8\}}+  \Z \alpha_{\{3,5\}} + \Z \alpha_{1} + \Z \alpha_{7}\]
viewed as a discrete subset in $\R^8$. The formulas preceding the example then give the theta series for the discrete subset corresponding to $B = \{2,4,6,8\}$ as a product of $4$ terms, $\vartheta_2(q^8) \vartheta_3(q^4) \vartheta_3(q^2) \vartheta_3(q^2)$. Repeating this process for each remaining fixed codeword above, we find that the theta series for the fixed sublattice $L^g$ is given by
{\small \begin{align*}  \theta_{L^g}(q) &= \vartheta_3(q^8)\vartheta_3(q^4)\vartheta_3(q^2)^2 + \vartheta_2(q^8)\vartheta_2(q^4)\vartheta_2(q^2)^2 + \vartheta_2(q^8)\vartheta_3(q^4)\vartheta_3(q^2)^2 + \vartheta_3(q^8)\vartheta_2(q^4)\vartheta_2(q^2)^2 \\
& = 1 + 14q + 30q^2 + 36q^3 + 62q^4 + 72q^5 + 68q^6 + 112q^7 + 126q^8 +98q^9 + O(q^{10}).
\end{align*} }
\end{Example} 

A similar process, together with equations \eqref{eqn: altconstruction} $-$ \eqref{eqn: gactionsupercode}, can also be used to give an explicit description of the theta functions of fixed sublattices of the Leech lattice and other lattices associated to super codes; see \cite{KT86}*{\S1} for further details and the repository~\cite{BK23} for an implementation in \texttt{Magma}.

\subsection{Sublattices fixed by subgroups of automorphisms} \label{subsec: sublatsfixedbysubgroups}
 
We now want to study sublattices fixed by subgroups of automorphisms of $G$ (up to conjugacy) and consider theta quotients associated to these sublattices, analogous to those in Section~\ref{subsec: thetaseries}. In order to compute fixed sublattices, we make use of the following lemma.

\begin{Lemma} \label{lem: LGintersection} Let $\{g_1, \dots, g_r\}$ be a set of generators for the subgroup $G$ of $\Aut L$. Then the sublattice of $L$ fixed by $G$ is the lattice \[ L^G = \bigcap_{g \in G} L^g = \bigcap_{i=1}^r L^{g_i}.\] 
\end{Lemma}

\begin{proof}
The first equality is by definition of a fixed sublattice, so we verify the latter. The forward containment $\bigcap_{g \in G} L^g\subseteq L^{g_1} \cap \cdots \cap L^{g_r} $ is immediate. Conversely, the set $L^{g_1} \cap \cdots \cap L^{g_r}$ is contained in $L^g$ for each $g \in G$ since any $g$ can be expressed as a word in the generators, and thus if $\lambda \in L$ is fixed by each of $g_1, \dots, g_r$, then it is fixed by $g$.
\end{proof}

Analogously to \eqref{fixedsublattice}, given $G \coloneqq \iota(\overline{G}) \subset \Aut L$ we may write
\begin{equation}\label{eqn: fixedsublatbyG}
L^G=\bigcup_{B\in C^{\overline{G}}}\left\{\textstyle \frac{1}{2} \alpha_B + \sum_i x_i \alpha_i:x_i=x_j
\mbox{ if $i$ and $j$ are in the same orbit under $\overline{G}$}\right\}.\end{equation}

\begin{Definition} \label{def: orbittype} Let $\overline{G}$ be a subgroup of $\Aut C$ and let $G \subseteq \Aut L$ be the image of $\overline{G}$ under the embedding $\iota$.  Since $\overline{G}$ is a subgroup of $S_N$, let $r_t$ denote the number of orbits of $\{1, \dots, N\}$ of size $t$ under the action of $\overline{G}$. We define the \defi{eta product associated to $L^G$} by \begin{equation} \eta_G(q) := \prod_t \eta(q^t)^{r_t}.\end{equation}
We further define the \defi{orbit type} of $G$ to be $\prod_t t^{r_t}$. 
\end{Definition}

The orbits of the action of $\overline{G}$ on the set $\{1, \dots, N\}$ and the codewords fixed by $\overline{G}$ have the following relationship.

\begin{Lemma} \label{lem: orbitdecomp} Let $\overline{G}$ be a subgroup of $\Aut C$ with orbit type $\prod_{i=1}^\ell t_i^{r_i}$ and let $m = \sum_{i=1}^{\ell}r_i$. Let $\{O_i: 1 \le i \le m\}$ denote the orbits of the action of $\overline{G}$. A codeword $B \in C$ is fixed by $\overline{G}$ if and only if the decomposition $\Omega = \cup_i O_i$ is a refinement of the decomposition $\Omega = B \cup (\Omega \smallsetminus B)$. 
\end{Lemma}

\begin{proof}
Suppose that $B \in C$ is fixed by the action of $\overline{G}$. Since $\overline{G}$ acts transitively on the orbit $O_i$, for each pair of elements $u,v \in O_i$, there exists $\gbar \in \overline{G}$ such that $\gbar(u) = v$. Thus $u,v$ are in the same disjoint cycle decomposition of $\gbar$. Since the disjoint cycle decomposition of any such $\gbar$ is a refinement of $\Omega = B \cup (\Omega \smallsetminus B)$, either $u,v$ are both in $B$ or both in $\Omega \smallsetminus B$. Since this holds for all elements of $O_i$, each $O_i$ is either a subset of $B$ or $\Omega \smallsetminus B$. The converse is immediate.
\end{proof}

We now give a proof of Theorem~\ref{thm: mainintro}, saving the discussion regarding non-monstrous replicable functions associated to non-cyclic subgroups $\overline{G} \subset \Aut C$ for Section~\ref{sec: repfunctionscodelats}. 

\begin{Theorem*}[Theorem \ref{thm: mainintro}] Let $C$ be a doubly even self-dual binary linear code of length $N$ and let $L$ be the associated code-lattice. Let $G \subset \Aut L$ be a subgroup of automorphisms in the image of $\Aut C$ under the natural embedding and let $L^G$ be the sublattice fixed by $G$. Then the lattice theta quotient $(\theta_{L^G}(q)/\eta_G(q))^{24/N}$ is a weakly holomorphic modular function. In the case that $G = \langle g \rangle$ is a cyclic subgroup, this recovers the usual lattice theta quotient associated to $g$. Moreover, this construction can produce non-monstrous replicable functions not necessarily associated to any individual automorphisms.
\end{Theorem*}

\begin{proof}
Let $L$ be the code-lattice associated to a code $C$ and let $G = \iota(\overline{G})$ be a subgroup of automorphisms of $L$ under the natural embedding. Assume that the orbit type of $\overline{G}$ is $\prod_{i=1}^\ell t_i^{r_i}$, so that $N = \sum_{i=1}^\ell t_i r_i$. We claim that both the theta series $\theta_{L^G}(q)$ and the eta product $(\eta_G(q))^{24/N}$ are modular forms of weight $k \coloneqq \frac{1}{2} \sum_{i=1}^\ell r_i$ and level $M \coloneqq N \prod_{i=1}^{\ell} t_i$, which we now explain.

The theta series of $L^G$ is determined by the codewords fixed by $\overline{G}$ as in \eqref{eqn: fixedsublatbyG}. Moreover, by Lemma~\ref{lem: orbitdecomp}, each fixed codeword $B$ has the property that the orbit decomposition of $\overline{G}$, written $\Omega = \cup_{i=1}^{m}O_i$, is a refinement of the decomposition $\Omega = B \cup (\Omega \smallsetminus B)$. Thus given $B \in C^{\overline{G}}$, after possibly reordering, we can write \[B = \cup_{i=1}^{s} O_i \qquad \text{and} \qquad \Omega \smallsetminus B = \cup_{i=s+1}^{m} O_i.\] Hence the theta series for the discrete subset in the union \eqref{eqn: fixedsublatbyG} associated to $B$ has the form \begin{equation} \label{eqn: thetaterm} \prod_{i=1}^s \vartheta_2(2|O_i|z) \prod_{i= s+1}^m \vartheta_3(2|O_i|z).\end{equation} 
Since $\vartheta_2(z) = 2\eta(2z)^2\eta(z)^{-1}$ and $\vartheta_3(2z) = \eta(2z)^5\eta(z)^{-2}\eta(4z)^{-2}$, a term of the form \eqref{eqn: thetaterm} is the eta quotient 
\begin{equation}\label{eqn: thetatermetaq} \prod_{i=1}^s 2 \, \eta(4|O_i|z)^2\eta(2|O_i|z)^{-1} \prod_{i=s+1}^m \eta(2|O_i|z)^5 \eta(|O_i|z)^{-2}\eta(4|O_i|z)^{-2}.\end{equation}
Now we can compute a level for each term \eqref{eqn: thetatermetaq} using \cite{Ono04}*{Theorem 1.64}.  Since $C$ is doubly even, we have $|B| = \sum_{i=1}^s |O_i| \equiv 0 \bmod{4}$. Therefore   \[ 8 \sum_{i=1}^s |O_i| - 2 \sum_{i=1}^s |O_i| + 10 \sum_{i=s+1}^m |O_i| - 10 \sum_{i=s+1}^m |O_i| = 6 \sum_{i=1}^s |O_i| = 6|B| \equiv 0 \bmod{24}. \] Since $N \equiv 0 \bmod{8}$ so that $2|O_i|$ and $4|O_i|$ divide $M$ for $1 \le i \le m$, we also have
\[ \sum_{i=1}^s \frac{2M}{4|O_i|} - \frac{M}{2|O_i|} + \sum_{i=s+1}^m \frac{5M}{2|O_i|} - \frac{2M}{|O_i|} - \frac{2M}{4|O_i|}
\equiv 0 \bmod{24}. \] This tells us that $\theta_{L^G}(q)$ is a sum of terms which are each {modular} of level $M$. From \eqref{eqn: thetatermetaq}, we also compute that the weight is
\[ \frac{1}{2}\left(s(2-1) + (m-s)(5-2-2)\right) = \frac{1}{2} m = \frac{1}{2} \sum_{i=1}^{\ell} r_i = k. \]
We note that this is the expected weight since the fixed sublattice $L^G$ has rank $m$. 

By construction, $\eta_G(q)$ also has weight $k$. Furthermore, the eta product $\eta_G(q)^{\frac{24}{N}}$ satisfies
\begin{align*} \frac{24}{N} \sum_{i=1}^\ell t_i r_i &\equiv 0 -\frac{24}{N} \cdot N \equiv 0 \bmod{24}, \quad \text{since $\sum_{i=1}^\ell t_i r_i = N$, and }\\
 \frac{24}{N} \sum_{i=1}^{\ell} \frac{M}{t_i} r_i &\equiv \frac{24}{N} \sum_{i=1}^{\ell} N t_1 \dots t_{i-1} \hat{t_i} t_{i+1} \dots t_\ell r_i \equiv 0 \bmod{24}, 
 \end{align*}
so it, too, has level $M$. Since $\eta(z)$ is analytic and nonvanishing on the upper half plane and meromorphic at each cusp, the theta quotient $(\theta_{L^G}(z)/\eta_G(z))^{24/N}$ is a weakly holomorphic modular function of level $M$, although this level may not be optimal.

Finally, this construction recovers the usual lattice theta quotient associated to $g\in \Aut L$ when $G = \langle g \rangle = \langle \iota(\gbar) \rangle$, since for a single element $\bar{g} \in S_N$, the disjoint cycle decomposition precisely describes the orbits of $\{1, \dots, N\}$ under the action of $\bar{g}$ and hence of the whole subgroup $\langle \gbar \rangle$.
\end{proof}

\begin{Example} Consider the subgroup $\overline{G} \subset \Aut \calH$ generated by 
\[ g_1 = (4,6)(5,7), \qquad g_2 = (4,7)(5,6), \quad \text{and} \quad g_3 = (1,3)(2,8),\] 
and let $G=\iota(\overline{G})\subset\Aut(E_8)$. We determine that the the orbits of the set $\Omega$ under the 
 action of $\overline{G}$ are $\{1,3\}, \{2,8\}, \{4,5,6,7\}$,
 hence the orbit type of $G$ is $2^2 4^1$ and $\eta_G(q) = \eta(q^2)^2\eta(q^4)$. Then, as in Example~\ref{ex:g_invariants}, the set of $\overline{G}$-fixed codewords are $\{ \emptyset, \Omega, \{1,2,3,8\}, \{4,5,6,7\} \}$ and, for example, the discrete subset associated to the codeword $B = \{1,2,3,8\}$ is $(\Z + \frac12) \alpha_{\{1,3\}} + (\Z + \frac12) \alpha_{\{2,8\}} + \Z \alpha_{\{4,5,6,7\}}$. Repeating this for each fixed codeword, we eventually find that the theta series of the $G$-fixed sublattice is 
\begin{eqnarray*} \theta_{E_8^G}(q) &=& \vartheta_3(q^4)^2\vartheta_3(q^8) + \vartheta_2(q^4)^2\vartheta_2(q^8) + \vartheta_2(q^4)^2\vartheta_3(q^8) + \vartheta_3(q^4)^2\vartheta_2(q^8) \\
& = & 1 + 6q + 12q^2 + 8q^3 + 6q^4 + 24q^5 + 24q^6 + 12q^8 + 30q^9  + O(q^{10}),
\end{eqnarray*}
the theta series for the lattice $A_1(2)^3$. Overall, the lattice theta quotient associated to $G$ is
\[ \left(\frac{E_8^G(q)}{\eta_G(q)}\right)^3 = q^{-1}(1 + 18q + 150q^2 + 780q^3 +2928q^4 + 88926q^5 + 24032q^6 + \dots),\]
and is a modular function of level $64$.
\end{Example} \medskip

The orbit type is a natural generalization of cycle type when one considers subgroups of lattice automorphisms in the image of those from the code. The cycle type of a general lattice automorphism is defined in terms of the decomposition of the characteristic polynomial into a product of cyclotomic polynomials \cite{GK19}*{Appendix A}. For a lattice automorphism in the image of an automorphism of the code, this coincides with the cycle type as defined in Section~\ref{subsec: linearcodes}.

\begin{Question}
Can one generalize orbit type to an arbitrary subgroup of lattice automorphisms? 
\end{Question}

\section{Lattice vertex operator algebras and automorphisms} \label{sec: latticeVOAs}
In this section, we review the construction of lattice vertex operator algebras (lattice-VOAs) which are VOAs built from even, positive definite lattices. We then recall automorphisms of VOAs and discuss those which are lifted from automorphisms of the lattice.
\subsection{Lattice-VOA construction} \label{subsec: Lattice-VOA construction}
Here we give a brief overview of the lattice-VOA construction in order to fix notation and ideas. For further details, we refer the reader to \cites{Mollerthesis, FLM88}. Given an even positive definite lattice $L$, we can form the \defi{twisted group algebra} $\mathbb{C}_{\epsilon}[L]$ spanned by the $\C$-basis $\{\mathfrak{e}_{\alpha}\}_{\alpha \in L}$ as follows. There exists a $2$-cocycle\\ $\epsilon: L\times L \to \{ \pm 1 \}$ such that $\epsilon(\alpha, \alpha)=(-1)^{\langle \alpha, \alpha \rangle / 2}$ and $\dfrac{\epsilon(\alpha, \beta)}{\epsilon(\beta, \alpha)}=(-1)^{\langle \alpha, \beta \rangle}$ for $\alpha, \beta\in L$ and where $\langle \cdot,\cdot \rangle$ is the inner product on $L$. Multiplication on $\C_{\epsilon}[L]$ is defined as $\mathfrak{e}_{\alpha}\mathfrak{e}_{\beta}= \epsilon(\alpha, \beta) \mathfrak{e}_{\alpha+\beta}$ and we define the weight of an element $\mathfrak{e}_{\alpha}$ to be ${\langle \alpha, \alpha \rangle / 2}$. Then the grading on $\C_{\epsilon}[L]$ is given in terms of weight.

Let $\mathfrak{h}= L\otimes_{\mathbb{Z}} \mathbb{C}[t,t^{-1}]$, then the Heisenberg algebra is defined as $\hat{\mathfrak{h}}=\mathfrak{h}\oplus \mathbb{C}\bf{k}$ where $\bf{k}$ is a central element. This is a Lie algebra, with Lie bracket for $h_1\otimes t^{n_1}, h_2\otimes t^{n_2}\in\hat{\mathfrak{h}}$ defined by linearly extending the bracket given by \[ [h_1\otimes t^{n_1}, h_2\otimes t^{n_2}]=\langle h_1,h_2\rangle n_1 \delta_{n_1+n_2} \bf{k},  \] where $h_1,h_2 \in \mathfrak{h}$ and $n_1,n_2\in\mathbb{Z}$. The Lie bracket of any element of $\hat{\mathfrak{h}}$ with $\bf{k}$ is zero. Let $h(n)=x\otimes t^n$ as an abbreviation, then we consider an $\hat{\mathfrak{h}}$-module, whose elements are of the form \[ h_k(-n_k) \dots h_1(-n_1)\bf{1}, \quad \text{for $n_1 \dots n_k\in\Z^+$.} \] 

The lattice-VOA $V_L$ is defined as a tensor product of this $\hat{\mathfrak{h}}$-module with the twisted group algebra and consists of elements of the form \[ h_k(-n_k) \dots h_1(-n_1)\textbf{1} \otimes  \mathfrak{e}_{\alpha}, \quad \text{for $n_1 \dots n_k \in \Z^{+}$ and $\mathfrak{e}_{\alpha}\in\C_{\epsilon}[L]$}. \] The weight of an element of this form is given by \[ n_1+ \dots + n_k+ {\langle \alpha, \alpha \rangle / 2}. \] 
One can equip this vector space with fields $Y$ and conformal vector $\omega$ and $V_L$ will be a vertex operator algebra with central charge equal to the rank of $L$.

\subsection{Automorphisms and fixed subVOAs} \label{sec:automorphismsVOAs} 

We now recall the definition of an automorphism of a VOA and discuss automorphisms of lattice-VOAs that arise as lifts of lattice automorphisms.
\begin{Definition}[VOA automorphism] For a VOA $V$, an automorphism of $V$ is a linear operator $h\colon V \to V$ such that $h\mathbf{1}=\mathbf{1}$, $h\omega=\omega$, and $h Y(v,z)h^{-1}=Y(hv,z)$ for all $v\in V$. We denote the group of automorphisms of $V$ by $\Aut V$. 
\end{Definition}

For an automorphism $h \in \Aut V$ of order $m$, we have that $V$ decomposes into eigenspaces \[ V=\bigoplus\limits_{k\in \mathbb{Z}/m\mathbb{Z}} V^k \] where $V^{k}=\{v\in V \mid hv=e^{2\pi i k/m}v\}$ for $0\leq k\leq m-1$. 
\begin{Definition}[Fixed subVOA] \label{def: fixedsubVOA} For a finite subgroup $H \subseteq \Aut V$ let $V^H$ be the set of $v\in V$ that are pointwise fixed under the action of $H$, i.e., \[V^H \coloneqq \{v\in V \mid hv=v \text{ for all } h\in H\}.\] By restricting the VOA structure from $V$ to $V^H$, $V^H$ has the structure of a vertex operator algebra.
\end{Definition}
 The main theorem of orbifold theory (see, for example  \cite{Mollerthesis}*{Theorem $4.1.5$} or
 \cites{CM16, Miyamoto, DM97}) asserts that if $V$ is strongly rational and $H$ is a finite solvable group of automorphisms of $V$, then the fixed-point VOA $V^H$ is strongly rational as well.

\subsection{Lifting lattice automorphisms} We consider automorphisms of lattice-VOAs that arise as lifts of automorphisms of the underlying lattice. An automorphism $g$ of an even positive definite lattice $L$ can be lifted to an automorphism $\hat{g}$ of the corresponding lattice-VOA $V_L$. Based on the lattice-VOA construction, we decompose the VOA automorphism to see its action on the twisted group algebra and on the Heisenberg algebra. We write $\hat{g}=g_{\mathfrak{h}} \otimes g_{\epsilon}$ where $g_{\mathfrak{h}}$ acts on the elements of the Heisenberg algebra  (denoted $h(-n)$ where $h\in \mathfrak{h}$ and $n\in \mathbb{Z}^+$) as $g(h(-n))=gh(-n)$. To define the action of $g$ on the elements of the twisted algebra, we require a function $u\colon L \to \{ \pm 1\}$ that is compatible with the $2$-cocycle $\epsilon$ defined above in the sense that, for $\alpha, \beta$ in $L$ we have
\[ \dfrac{\epsilon(\alpha, \beta)}{\epsilon(g\alpha, g\beta)}= \dfrac{u(\alpha) u(\beta)}{u(\alpha+ \beta)}.\] 

\noindent We can then define the action of $\hat{g}$ on $\mathbb{C}_{\epsilon}[L]$ as $\hat{g}(\mathfrak{e}_{\alpha})=u(\alpha) \mathfrak{e}_{g\alpha}$. 
The action of $g^j$ on the Heisenberg algebra is straightforward, but it takes some care to define the action of $\hat{g}$ on $\mathbb{C}_{\epsilon}[L]$. In particular, $\hat{g}^j(\mathfrak{e}_{\alpha})=u(\alpha)u(g\alpha)\dots u(g^{j-1} \alpha) \mathfrak{e}_{g^j\alpha}$. To simplify notation, we define
\begin{equation*}w_j(\alpha):=u(\alpha)u(g\alpha)\dots u(g^{j-1}\alpha).\end{equation*}

We will take $u$ to be a \defi{standard lift} of $g$, so that $u(\alpha)=1$ for each $\alpha \in L^g$. However, since $u(g^j\alpha)$ is not necessarily equal to $1$ if $j > 1$,  we use the following theorem to compute $w_j(\alpha)$.
\begin{Theorem}[\cite{Borcherds92} \cite{Mollerthesis}] Let $g$ be an automorphism of $L$ of order $m$ and $\hat{g}$ its lift to $V_L$, then for all $j\in\Z_{\geq 0}$
\[w_j(\alpha)= u\left(\sum\limits_{i=0}^{j-1} g^i\alpha\right)\cdot \begin{cases} 1 & \text{ if $m$ or $j$ is odd}, \\
(-1)^{\langle \alpha, g^{j/2}\alpha \rangle}  & \text{if $m$ and $j$ are even.}
\end{cases}\]
In particular, $\sum_{i=0}^{j-1} g^i\alpha$ is in $L^g$ and thus if $u$ is the standard lift, this factor is $1$. Hence $w_j$ defines a homomorphism $w_j \colon L^{g^j} \to \{\pm{1}\}$.
\end{Theorem}

 In particular, if both the order $m$ of the automorphism $g$ and the power $j$ of $g^j$ are even, then $w_j(\alpha)$ need not equal $1$. If $g \in \Aut (L)$ has order $m$ and there is an $\alpha$ such that $w_j(\alpha) \ne 1$, then the lift $\hat{g} \in \Aut(V_L)$ has order $2m$.
\begin{Corollary}[Order doubling] \label{cor: orderdoubling} Let $\hat{g} \in \Aut(V_L)$ be a standard lift of an order $m$ automorphism $g \in \Aut(L)$. If $m$ is odd, then $\hat{g}$ has order $m$. If $m$ is even, then $\hat{g}$ has order $m$ if $\langle \alpha, g^{m/2}\alpha \rangle \in 2\Z$ for all $\alpha \in L$, and order $2m$ otherwise.
\end{Corollary}

\section{Characters of fixed subVOAs} \label{sec: charsfixedsubVOAS}
We now give the definition of the characters of the fixed subVOAs and an example of how to compute them.
For a lattice-VOA $V_L$, we take the fixed sublattice under a finite cyclic group of automorphisms $\langle \hat{g} \rangle$ of order $n$. By the classification of irreducible modules in \cite{Mollerthesis}, there are exactly $n^2$ irreducible $V^{\langle \hat{g} \rangle}$-modules,  namely \[ W^{(\ell,j)}=\{ w\in V(\hat{g}^\ell)\mid \phi_\ell(\hat{g})v=e^{2\pi \Im j/n}v \}\] 
where $\phi_{\ell}$ is a representation of $\langle  \hat{g} \rangle$ on the vector space $V_L(\hat{g}^{\ell})$ that is unique up to an $n$-th root of unity (see \cite{Mollerthesis}*{Proposition 4.2.3} or \cite{DLM}). Recall that $V=\oplus V_n$ is graded by the $L_0$ eigenvalues, so we define $\text{tr}(q^{L_0-c/24}|V) \coloneqq q^{-c/24} \sum_n \text{dim}{V_n}\,q^n$.

\begin{Definition}[Twisted trace functions] \label{twistedtracefunctions} Following \cites{Miyamoto, CM16, MT04}, let
\begin{equation} T(v,\ell ,j, \tau) \coloneqq \text{tr}(o(v)\phi_\ell(\hat{g}^j)q^{L_0-\frac{c}{24}}\mid V(\hat{g}^\ell)), \end{equation}
where for $v\in V_k$, we have $o(v):=v(k-1)$ which leaves each homogeneous space $V_k$ invariant and can be linearly extended to finite sums $u=\sum_k o(v_k)$.
\end{Definition}

\begin{Definition}
The \defi{character} of a vertex operator algebra $V=\bigoplus_n V_n$ of central charge $c$ is defined as 
\[ \Ch V \coloneqq \sum\limits_n \text{dim}(V_n)q^{n-c/24}.\]
\end{Definition}
 
To compute the character of $V_L$ fixed by $\langle \hat{g} \rangle$, we sum over the traces of $\hat{g}^j$ acting on $V_L$ for all $\hat{g}^j \in \langle \hat{g} \rangle$. In the particular case when $v=\mathbf{1}$, the twisted character for the action of $\hat{g}^j$ on the twisted module $V(\hat{g}^\ell)$ is
\begin{equation} \label{eqn: Tij} T(\ell,j,\tau)=T(\mathbf{1},\ell, j, \tau)= \text{tr}(\phi_\ell(\hat{g}^j)q^{L_0-\frac{c}{24}}\mid V(\hat{g}^\ell)). 
\end{equation}

Taking $\phi_0(\hat{g}):=\hat{g}$ for all $\hat{g}\in \Aut V_L$\footnote{Since $V_L$ is untwisted, this is a suitable choice for $\phi_0$ (see, for example Remark $4.2.2$ of \cite{Mollerthesis})}, we observe from \eqref{eqn: Tij} that \[T(0,j,\tau)=\text{tr}(\phi_0(\hat{g}^j)q^{L_0-\frac{c}{24}}\mid V_L)=\text{tr}(\hat{g}^jq^{L_0-\frac{c}{24}}\mid V_L),\] which is the trace of $\hat{g}^j$ acting on $V_L$. Note that in terms of the twisted modules above, we have $\text{Ch} V_L^{\langle \hat{g} \rangle} =\text{Ch}{W^{(0,0)}}(\tau)$. Thus
\begin{equation} \text{Ch} V_L^{\langle \hat{g} \rangle}=\frac{1}{n} \sum\limits_{j\in \mathbb{Z}/n\mathbb{Z}} T(0,j,\tau)
= \frac{1}{n} \sum\limits_{j\in \mathbb{Z}/n\mathbb{Z}}\text{tr}(\hat{g}^jq^{L_0-\frac{c}{24}}\mid V_{L})=\frac{1}{n} \sum\limits_{j\in \mathbb{Z}/n\mathbb{Z}}\frac{ \theta_{L^{g^j},w_j}(\tau)}{\eta_{g^j}(\tau)},
\end{equation}
with $\theta_{L^{g^j},w_j}(\tau)=\sum\limits_{\alpha \in L^{g^j}} w_j(\alpha)q^{{\langle \alpha, \alpha \rangle}/{2}}$ and $w_j(\alpha)=u(\alpha)u(g\alpha) \dots u(g^{j-1}\alpha)$, as before.

By Corollary \ref{cor: orderdoubling}, when $j$ or $n$ is odd, the theta series $\theta_{L^{g^j},w_j}(\tau)$ is the usual lattice theta series $\theta_{L^g}(\tau)$ of the sublattice fixed by $g^j$. When $j$ and $n$ are both even, the theta series $\theta_{L^{g^j},w_j}(\tau)$ is not necessarily equal to $\theta_{L^{g^j}}(\tau)$ since $w_j(\alpha)$ need not be $1$. In this case, since $w_j$ defines a homomorphism into $\{\pm{1}\}$, we can decompose $L^{g^j}$ as a direct sum 
\begin{equation}
     L^{g_j} = w_j^{-1}(\{1\}) \oplus w_j^{-1}(\{-1\}) = L_0^{g^j} \oplus w_j^{-1}(\{-1\}).
\end{equation}
We call $L_0^{g^j} \coloneqq w_j^{-1}(\{1\}) = \ker(w_j)$ the \defi{kernel for order doubling}. Its complement $w_j^{-1}(\{-1\})$ has theta series $\theta_{L^{g^j}}(\tau) - \theta_{L_0^{g^j}}(\tau)$. Separating the terms of $\theta_{L^{g^j},w_j}(\tau)$ according the value of $w_j(\alpha)$ in each term, we have $\theta_{L^{g^j},w_j}(\tau) = \theta_{L_0^{g^j}}(\tau)-\theta_{w_j^{-1}(\{-1\})}(\tau)$. In summary, we have

\begin{equation} \label{eqn: modifiedlatticetheta}
 \theta_{L^{g^j},w_j}(\tau)  =  \begin{cases}
    \theta_{L^{g^j}}(\tau), & \text{if one of $j$ or $n$ is odd} \\
    2\theta_{L_0^{g^j}}(\tau) - \theta_{L^{g^j}}(\tau), & \text{if both $j$ and $n$ are even}.\\
    \end{cases}
\end{equation}

\noindent In the case of $E_8$, we prove that the kernel for order doubling is the $D_8$ lattice.

\begin{Lemma} \label{lem: kerorddoubE8} Let $L = E_8$ and suppose $g \in \Aut(L)$ with even order $n$ whose lift $\hat{g}$ has order $2n$ in $\Aut V_L$. Then when $j = n$, we have $\theta_{L^{g^j},w_j}(\tau) = 2\theta_{D_8}(\tau) - \theta_{E_8}(\tau)$.
\end{Lemma} 

\begin{proof} The kernel $\ker(w_j)$ is always a full rank sublattice of $L^{g^j}$, the latter of which is $L$ itself in the case that $j = n$ since $g$ has order $n$. By assumption, $\hat{g}$ has order doubling, thus $\ker(w_j) \subsetneq E_8$. Moreover, since $D_8 = \{ \sum x_i e_i : \sum_i x_i \equiv 0 \bmod{2}\} \subset E_8$, the norm $\langle \alpha, g \alpha \rangle \in 2\Z$ for all $\alpha \in D_8$, hence $D_8 \subset \ker(w_j)$. Finally, since $D_8$ has index $2$ in $E_8$, we must therefore have $\ker(w_j) = D_8$ and thus the claim follows from (\ref{eqn:  modifiedlatticetheta}).
\end{proof}

\begin{Example} Consider an element $\gbar \in \Aut \calH$ of cycle type $1^22^14^1$ and let $g = \iota(\gbar)$.  By applying Corollary~\ref{cor: orderdoubling}, we can determine that the lift of $g$ to $\hat{g} \in \Aut V_{E_8}$ has order doubling, i.e., $\text{ord}(\hat{g}) = 8$. For odd $j$, since the fixed sublattices $L^g$ and $L^{g^3}$ are the same, we find 
\[ T(0,j, \tau)=q^{-1/3}(1 + 16q + 64q^2 + 192q^3+ 510q^{4} + 1216q^{5} + 2688q^{6} + \cdots).\]
We handle even powers of $g$ individually. To compute $T(0,0, \tau) = T(0, 8,\tau)$, 
since $g^8$ is the identity, this gives $\theta_{L^{g^8}}(\tau)= \theta_{L^{\text{id}}}(\tau)=\theta_{E_8}(\tau)$. Thus we have
\[ T(0,0,\tau) =q^{-1/3}(1 +248q + 4124q^2 +34752q^3 +213126q^{4} + 1057504q^{5} + 4530744q^{6} + \cdots). \]
Since $g^2$ has cycle type $1^4 2^2$, the fixed sublattice $L^{g^2}$ has rank $6$ and has theta series
 \[\theta_{L^{g^2}}(\tau)= 1 + 60q + 252q^2 + 544q^3 + 1020q^4 + 1560q^5 + 2080q^6 + 3264q^7 + 4092q^8 +\cdots .\] 
The kernel of $w_2$ is a full rank sublattice of $L^{g^2}$, 
 and the theta series of $L_0^{g^2} = \ker(w_2)$ is 
 \[\theta_{L_0^{g^2}}(\tau) = 1 + 28q + 124q^2 + 288q^3 + 508q^4 + 728q^5 + 1056q^6 + 1728q^7 +2044q^8 + \cdots.\]
We then use \eqref{eqn:  modifiedlatticetheta} to compute the theta series. Following the same process for $j = 4$ and $j =6 $, we find
\begin{align*}
 T(0,2,\tau)&=q^{-1/3}(1 -4q^2 + 6q^4 -8q^6 +17q^8 - 28q^{10} + \cdots), \\
T(0,4,\tau)&=q^{-1/3}(1 -8q + 28q^2 -64q^3 +134q^{4} - 288q^{5} + 568q^{6} + \cdots),\mbox{ and} \\
T(0,6,\tau)&=q^{-1/3}(1 -4q^2 + 6q^4 -8q^6 +17q^8 - 28q^{10} + \cdots).\end{align*}
Finally, by combining the previous calculations, we have 
\[\text{Ch} V_L^{\langle \hat{g} \rangle}= q^{-1/3}(1 + 38q + 550q^2 + 4432q^3 + 26914q^4 + 132760q^5 + 567756q^6 + \cdots ). \]
\end{Example}

\section{Replicable functions associated to code-lattices} \label{sec: repfunctionscodelats}
In this section, we prove results about the modular functions associated to sublattices of even unimodular code-lattices fixed by subgroups of automorphisms of the code. In particular, motivated by Monstrous Moonshine, we consider the question of when these modular functions are replicable functions, and recover classical results regarding the functions that arise as theta quotients associated to sublattices of the $E_8$ lattice fixed by a single automorphism. More generally, we show that when one considers sublattices fixed by \emph{subgroups} of automorphisms, it is possible to recover additional replicable lattice theta quotients. We make this explicit in Section~\ref{subsub: leechnewreplicable} for the Leech lattice. Furthermore, we characterize those replicable theta quotients associated to orbit types that can exist for code-lattices of arbitrary rank, and provide data for lattices of rank at most $24$.

In Figures \ref{fig: repetasE8} -- \ref{fig: replicablenoncyclicLeech}, capital letters indicate monstrous functions for which ATLAS notation is used as in \cite{CNMoonshine}, while lowercase letters indicate non-monstrous replicable functions.

\subsection{Lattices fixed by an automorphism} 
The relationship between the monstrous moonshine functions and graded trace functions of automorphisms acting on lattice-VOAs (although not stated this way) has been investigated for individual lattice automorphisms $g$. In particular, when $L$ is the Leech lattice, Conway and Norton \cite{CNMoonshine} speculated that for any $g\in \Aut L$, the lattice theta quotient $\theta_{L^g}(\tau)/\eta_g(\tau)$ is a monstrous moonshine function. This is in fact not the case; there are 15 conjugacy classes of elements of $\Aut L$ such that $\theta_{L^g}(\tau)/\eta_g(\tau)$ are not hauptmoduln \cite{L89}. However, Kondo and Tasaka considered only automorphisms of $L$ which are in the image of an automorphism $\Aut \mathcal{G}$ under a natural embedding \cite{KT86}. In this case they show that all of the lattice theta quotients $\theta_{L^g}(\tau)/\eta_g(\tau)$ are monstrous moonshine functions. This leads naturally to the following question.

\begin{Question} Let $C$ be a doubly-even self dual code and let $L$ be the corresponding code-lattice. If one restricts to the trace functions $\theta_{L^g}(\tau)/\eta_g(\tau)$ for $g \in \Aut L$ in the image of $\bar{g}\in \Aut C$, are these all replicable functions? If not, to what extent does this fail?
\end{Question}

In the case when $L$ is the $E_8$ lattice, the automorphisms of $L$ that give hauptmoduln are classified in \cite{CLL92}*{Main Theorem}.  In analogy with \cite{KT86}, we investigate the lattice theta quotients $(\theta_{L^g}(\tau)/\eta_g(\tau))^3$ when $g$ is an automorphism of $E_8$ in the image of $\Aut \mathcal{H}$ under the embedding $\iota$, defined in equation \eqref{eqn: iota}. 
In particular, we find they are equal to monstrous moonshine functions or non-monstrous replicable functions for certain conjugacy classes. As in Theorem~\ref{thm: mainintro}, we take the cube to ensure that the theta quotients have desired transformation properties and are in fact modular functions. We summarize our findings below. \medskip

\begin{center}
\def\arraystretch{1.1}
\begin{tabular}{|l| l| l|}
\hline
Cycle Type of $\overline{g}$ & Representative & $(\theta_{L^g}(q)/\eta_g(q))^3$ \\
\hline
$1^8$ & $\text{Id}(G)$ & $T_{1A}(q)$ \\
\hline
$1^2 \, 3^2$ & $(1, 5, 2)(3, 7, 8)$& $T_{3A}(q)$  \\
\hline
$2^4$ & $(1, 7)(2, 4)(3, 8)(5, 6)$ & $T_{4A}(q)$  \\
\hline
$1^1 \, 7^1$ & $(1, 3, 7, 8, 5, 4, 2)$& $T_{7A}(q)$   \\
\hline
$4^2$ & $(1, 3, 7, 8)(2, 5, 4, 6)$ & $T_{8B}(q)$  \\
\hline
$2^1 \, 6^1$ & $(1, 3, 7, 8, 2, 6)(4, 5)$ & $T_{6b}(q)$  \\
\hline
\end{tabular}
\captionof{figure}{List of cycle types of automorphisms of $\calH$ that produce replicable lattice theta quotients for $E_8$}
\label{fig: repetasE8}\end{center}

\begin{Proposition} \label{prop: replicableE8individual}
Let $\bar{g}$ be an automorphism of the Hamming code $\mathcal{H}$. Then $(\theta_{L^g}(z)/\eta_g(z))^3$ is a replicable function if and only $\bar{g}$ is in one of the 6 conjugacy classes listed in Figure~\ref{fig: repetasE8}.
\end{Proposition}

\begin{proof}
We first note that one could obtain the result from \cite{CLL92}*{Main Theorem} by restricting to only those elements of $\Aut E_8$ in the image of $\Aut \calH$ under the natural embedding. We proceed by considering all $11$ conjugacy classes of elements $\overline{g}$ in $\Aut \calH$. Using Definition~\ref{def:ReplicableFunction}, one can verify that the the associated theta quotients $\left(\theta_{L^g}(z)/\eta_g(z)\right)^3$ are not replicable for the following 
conjugacy classes: $2^4$ with representative $(1, 2)(3, 8)(4, 7)(5, 6)$, $1^4 \, 2^2$ with representative $(1, 6)(7, 8)$, $4^2$ with representative $(1, 5, 2, 6)(3, 7, 8, 4)$, and $1^2 \, 2^1 \, 4^1$ with representative $(1, 7, 8, 6)(4, 5)$. For the remaining conjugacy classes, one can compute the corresponding fixed sublattices and apply Proposition~\ref{prop: replicable_eta_quos} to recover the replicable functions in Figure~\ref{fig: repetasE8}. 
\end{proof}

\begin{remark}
The cycle types in $\Aut \calH$ whose lattice theta quotients are not replicable functions, namely $\{2^4, 1^42^2, 4^2, 1^22^14^1 \}$, coincide with the cycle types whose lifts to $V_{E_8}$ have order doubling. Note, however that while the cycle types are the same, the conjugacy classes are not necessarily the same. 
\end{remark}

\subsection{Lattices fixed by subgroups of automorphisms}

Recall that in Definition~\ref{def: orbittype} we extended the notion of cycle type of an element $g \in \Aut L$ to an orbit type of a subgroup $G \subset \Aut L$, thereby allowing us to define an eta product $\eta_G(q)$ associated to a subgroup $G$.

\begin{Question} \label{question: subgroups}
    If one considers lattice theta quotients for lattices fixed by \emph{subgroups} of $\Aut C$ rather than individual elements of $\Aut C$, which, if any, of these functions are replicable? Will any additional replicable functions arise that were not realized as theta quotients from individual elements?
\end{Question}

\noindent We first complete this analysis for subgroups of $\Aut \mathcal{H}$.

\begin{center}
\def\arraystretch{1.1}
\begin{tabular}{|c| l| c|}
\hline
Orbit Type & Generators of representative subgroup & $(\Theta_{L^G}(q)/\eta_G(q))^3$ \\
\hline
$1^8$ & \text{id} & $T_{1A}(q)$ \\
\hline

\multirow{2}{*}{$1^2\, 3^2$} & $(1, 5, 2)(3, 7, 8)$&  \multirow{2}{*}{$T_{3A}(q)$}  \\ \cline{2-2}
 & $(1, 4, 3)(5, 8, 7),
(1, 3)(5, 7)$ &  \\
\hline
\multirow{2}{*}{$2^4$} & $(1, 7)(2, 4)(3, 8)(5, 6)$ &  \multirow{2}{*}{$T_{4A}(q)$}  \\
\cline{2-2}
& $(1, 5)(2, 6),
(3, 8)(4, 7)$ & \\
\hline

\multirow{3}{*}{$1^1 \, 7^1$} & $(1, 3, 7, 8, 5, 4, 2)$&  \multirow{3}{*}{$T_{7A}(q)$}   \\
\cline{2-2}
 & $(1, 4, 3)(5, 8, 7),
(1, 7, 3, 4, 6, 5, 8)$ & \\
\cline{2-2}
& $(1, 4)(3, 6),
(1, 3, 8, 2)(4, 5)$ & \\
\hline
\multirow{5}{*}{$4^2$} & $(1, 3, 7, 8)(2, 5, 4, 6)$ & \multirow{5}{*}{$T_{8B}(q)$}   \\
\cline{2-2}
& $(2, 5)(3, 4),
(1, 8)(2, 5)(3, 4)(6, 7),
(1, 5)(2, 8)(3, 7)(4, 6)$ & \\
\cline{2-2}
 & $(1, 2)(3, 7)(4, 8)(5, 6),
(1, 7)(2, 3)(4, 5)(6, 8)$ & \\
\cline{2-2}
& $(1, 2)(3, 7)(4, 8)(5, 6),
(1, 7)(2, 3)(4, 5)(6, 8),
(1, 2, 3)(4, 6, 5)$ & \\
\cline{2-2}
& 
\begin{tabular}{l@{}}
$(1, 2, 5)(3, 7, 4),
(2, 5)(3, 4), (1, 8)(2, 5)(3, 4)(6, 7),$\\
$
(1, 2)(3, 6)(4, 7)(5, 8)$
\end{tabular}
 & \\

\hline
\multirow{6}{*}{$2^1 \, 6^1$} & $(1, 3, 7, 8, 2, 6)(4, 5)$ & \multirow{6}{*}{$T_{6b}(q)$}  \\
\cline{2-2}
 & $(1, 2)(3, 7)(4, 8)(5, 6),
(1, 6, 8)(2, 4, 5)$ & \\
\cline{2-2}
& $(4, 6)(5, 7),
(2, 7, 5)(3, 6, 4),
(1, 8)(2, 3)(4, 5)(6, 7)$ & \\
\cline{2-2}
& $(1, 2, 5)(3, 7, 4),
(2, 4)(3, 5),
(1, 7)(2, 4)(3, 5)(6, 8),
(1, 7)(2, 4)$ & \\
\cline{2-2}
& $(1, 2, 5)(3, 7, 4),
(2, 4)(3, 5),
(1, 7)(2, 3)(4, 5)(6, 8),
(1, 7)(2, 4)$ & \\
\cline{2-2}
& 
\begin{tabular}{l@{}}
$(2, 4)(3, 5),
(1, 7)(2, 4)(3, 5)(6, 8),
(2, 5)(3, 4),
(1, 7)(2, 4),$\\
$(1, 5, 2)(3, 4, 7)$
\end{tabular}
 & \\
\hline 
\end{tabular}
\captionof{figure}{The replicable lattice theta quotients associated to subgroups of automorphisms of the Hamming code $\calH$ embedded as automorphisms of $E_8$.}
\label{fig: repfcnsforE8}\end{center}

\begin{Proposition} \label{prop: subgroupsAutHamming}
Let $\overline{G}$ be a subgroup of automorphisms of the Hamming code $\mathcal{H}$ whose image is $G := \iota(\overline{G})$ in the automorphism group of $L \coloneqq E_8$. Then $(\theta_{L^G}(z)/\eta_G(z))^3$ is a replicable function if and only if the subgroup $\overline{G}$ is isomorphic to one of the 19 conjugacy classes of subgroups in Figure~\ref{fig: repfcnsforE8}.
\end{Proposition}

\begin{proof}
The proof proceeds in the same manner as Proposition~\ref{prop: replicableE8individual} except we consider all 95 conjugacy classes of subgroups of $\Aut \calH$. In particular, for each subgroup listed above we compute the fixed sublattice and apply Proposition~\ref{prop: replicable_eta_quos}, otherwise we determine the theta quotient is not replicable via Definition~\ref{def:ReplicableFunction}.
\end{proof}

\subsection{Comparisons of theta quotients across lattices of varying rank}\label{subsection:varyingrank}
We now gather results regarding the replicable functions that frequently appear as lattice theta quotients for even unimodular code-lattices $L$ as the rank $N$ varies. Since $N \equiv 0 \bmod{8}$, some of the orbit types that can appear for any such $N$ are $1^N$, $2^{N/2}$, $4^{N/4}$, $1^{N/4} \, 3^{N/4}$, $2^{N/8} \, 6^{N/8}$, $1^{N/8} \, 7^{N/8}$, and $8^{N/8}$, since these correspond to the partitions of $8$ where each distinct part appears the same number of times. 

The orbit type of a subgroup $G \subset \Aut(L)$ is not enough to characterize the fixed sublattice and hence the corresponding lattice theta quotient. Therefore, we state results about pairs of an isomorphism class of fixed sublattice together with an orbit type listed above which must give rise to replicable theta quotients, with a focus on those pairs that arise for the $E_8$ lattice and appear in higher rank lattices. Moreover, in Example \ref{ex: Leechnoncyc} and Figure \ref{fig: replicablenoncyclicLeech}, we recover additional replicable theta quotients associated to the Leech lattice, and in particular one from a non-cyclic subgroup with orbit type $2^3 \, 6^3$. 

\begin{Proposition} \label{prop: replicable_eta_quos} Let $G$ be a subgroup of $\Aut(L)$ and $L^G$ be the fixed sublattice.
\begin{enumerate}[itemsep = .4em]
    \item \label{part: 4A} If $G$ has orbit type $2^{N/2}$ with $L^G \cong A_1(2)^{N/2}$, then $(\theta_{L^G}(q)/\eta_{G}(q))^{24/N}= T_{4A}(q)$.
   \item \label{part: 8B}  If $G$ has orbit type $4^{N/4}$ with $L^G \cong A_1(4)^{ N/4}$, then $(\theta_{L^G}(q)/\eta_{G}(q))^{24/N}= T_{8B}(q)$.
   \item \label{part: 16a} If $G$ has orbit type $8^{N/8}$ with $L^G \cong A_1(8)^{N/8}$, then $(\theta_{L^G}(q)/\eta_{G}(q))^{24/N} = T_{16a}(q)$.
   \item \label{part: 3A} If $G$ has orbit type $1^{N/4} \, 3^{N/4}$ with $L^G \cong A_2^{N/4}$, then  $(\theta_{L^G}(q)/\eta_{G}(q))^{24/N} = T_{3A}(q)$.
   \item \label{part: 6b} If $G$ has orbit type $2^{N/8} \, 6^{N/8}$ with $L^G \cong A_2(2)^{ N/8}$, then $(\theta_{L^G}(q)/\eta_{G}(q))^{24/N} = T_{6b}(q)$.
  \item \label{part: 12A} If $G$ has orbit type $2^{N/8} \, 6^{N/8}$ with $\theta_{L^G}(q)= \left(\theta_{A_1}(q^2) \, \theta_{A_1}(q^6)\right)^{N/8}$, then the lattice theta quotient $(\theta_{L^G}(q)/\eta_{G}(q))^{24/N} = T_{12A}(q)$.
   \item \label{part: 7A} If $G$ has orbit type $1^{N/8} \, 7^{N/8}$ with $L^G \cong K$, the Kleinian lattice with Gram matrix $\left( \begin{smallmatrix} 2 & 1 \\ 1 & 4 \end{smallmatrix} \right)$, then  $(\theta_{L^G}(q)/\eta_{G}(q))^{24/N}=T_{7A}(q)$.
\end{enumerate}
\end{Proposition}

\begin{proof} 
\eqref{part: 4A} The theta quotient for the fixed sublattice is given by
\begin{align*} \left(\frac{\theta_{L^G}(q)}{\eta_G(q)}\right)^{\frac{24}{N}} = \left(\frac{\theta_{A_1}(q^2)}{\eta(q^2)}\right)^{\frac{N}{2} \cdot \frac{24}{N}}  = \left( \frac{\eta(q^2)^4}{\eta(q)^2\eta(q^4)^2} \right)^{12}  = \left( \frac{\eta(q^2)^2}{\eta(q)\eta(q^4)} \right)^{24} = T_{4A}(q) \end{align*}
where the third equality follows from the fact that the theta function for the lattice $A_1(2)$ is the Jacobi theta function $\vartheta_3(q^2) = \eta(q^2)^5/(\eta(q)\eta(q^4))^2$. \smallskip

\noindent \eqref{part: 8B} This result follows from the proof of \eqref{part: 4A} since the lattice theta quotient is
\[ \left(\frac{\theta_{L^G}(q)}{\eta_G(q)}\right)^{\frac{24}{N}} = \left(\frac{\theta_{A_1}(q^4)}{\eta(q^4)}\right)^{\frac{N}{4} \cdot \frac{24}{N}} = \left(\frac{\theta_{A_1}(q^4)}{\eta(q^4)}\right)^6 =  T_{4A}(q^2)^{1/2} = T_{8B}(q).\]

\noindent \eqref{part: 16a} This result similarly follows from the proof of \eqref{part: 4A} since the lattice theta quotient is
\[ \left(\frac{\theta_{L^G}(q)}{\eta_G(q)}\right)^{\frac{24}{N}} = \left(\frac{\theta_{A_1}(q^8)}{\eta(q^8)}\right)^{\frac{N}{8} \cdot \frac{24}{N}} = \left(\frac{\theta_{A_1}(q^8)}{\eta(q^8)}\right)^3 =  T_{4A}(q^4)^{1/4} = T_{16a}(q).\]

\noindent \eqref{part: 3A} The theta quotient for the given fixed sublattice has the form 
\[ \left(\frac{\theta_{L^G}(q)}{\eta_G(q)}\right)^{\frac{24}{N}} = \left(\frac{\theta_{A_2}(q)}{\eta(q)\eta(q^3)}\right)^{\frac{N}{4} \cdot \frac{24}{N}} = \left(\frac{\theta_{A_2}(q)}{\eta(q)\eta(q^3)}\right)^6. \]
Since $\theta_{A_2}(q)^6 = \theta_{A_2^6}(q) = \theta_{A_2^3}(q)^2= \left(\frac{\eta(q)^{12} + 27\eta(q^3)^{12}}{(\eta(q)\eta(q^3))^3}\right)^2$, we find
\[ \left(\frac{\theta_{A_2}(q)}{\eta(q)\eta(q^3)}\right)^6 = \left(\frac{\theta_{A_2^3}(q)}{\eta(q)^3 \eta(q^3)^3}\right)^2 = \left(\frac{\eta(q)^6}{\eta(q^3)^6} + 27\frac{\eta(q^3)^6}{\eta(q)^6}\right)^2 = T_{3A}(q).\]

\noindent \eqref{part: 6b} This result follows from the proof of \eqref{part: 3A} since the lattice theta quotient is given by
\[ \left(\frac{\theta_{L^G}(q)}{\eta_G(q)}\right)^{\frac{24}{N}} = \left(\frac{\theta_{A_2}(q^2)}{\eta(q^2)\eta(q^6)}\right)^{\frac{N}{8} \cdot \frac{24}{N}} = \left(\frac{\theta_{A_2}(q^2)}{\eta(q^2)\eta(q^6)}\right)^3 =  T_{3A}(q^2)^{1/2} = T_{6b}(q).\] 

\noindent \eqref{part: 12A} The theta quotient for the described fixed sublattice has the form 
\begin{align*} \left(\frac{\theta_{L^G}(q)}{\eta_G(q)}\right)^{\frac{24}{N}} = \left(\frac{\theta_{A_1}(q^2) \theta_{A_1}(q^6)}{\eta(q^2)\eta(q^6)}\right)^{\frac{N}{8} \cdot \frac{24}{N}} 
= \left(\frac{\eta(q^2)^2\eta(q^6)^2}{\eta(q)\eta(q^4)\eta(q^3)\eta(q^{12})}\right)^6 =  T_{12A}(q). \end{align*} 
\eqref{part: 7A} Since the theta function of the Kleinian lattice $K$ is \[\theta_K(q) = \frac{(\eta(q) \eta(q^7))^3 + 4(\eta(q^2) \eta(q^{14}))^3}{\eta(q)\eta(q^2)\eta(q^7)\eta(q^{14})},\]
we find that the theta quotient for the given fixed sublattice has the form 
\begin{align*} \left(\frac{\theta_{L^G}(q)}{\eta_G(q)}\right)^{\frac{24}{N}} = \left(\frac{\theta_{K}(q)}{\eta(q)\eta(q^7)}\right)^{\frac{N}{8} \cdot \frac{24}{N}} 
\left(\frac{\eta(q) \eta(q^7)}{\eta(q^2)\eta(q^{14})} + 4\left(\frac{(\eta(q^2) \eta(q^{14})}{\eta(q)\eta(q^7)}\right)^2 \right)^3 = T_{7A}(q).
\end{align*}
\end{proof}

\begin{Example} \label{ex: Leechnoncyc} The automorphism group of the Golay code $\mathcal{G}$ does not contain any elements of cycle type $2^3 6^3$. However, it does contain four conjugacy classes of non-cyclic subgroups with this \emph{orbit type}. A \texttt{Magma} computation shows that the embedding of these subgroups into the automorphism group of $\Lambda_{24}$ all have fixed sublattices isomorphic to $A_2(2)^3$. Thus by Proposition~\ref{prop: replicable_eta_quos}, the corresponding lattice theta quotients are all equal to the non-monstrous replicable function $T_{6b}(q)$. For explicit generators or additional replicable functions that arise for orbit types associated to other conjugacy classes of non-cyclic subgroups, see Section~\ref{subsub: leechnewreplicable}. 
\end{Example}

\begin{remark} Proposition \ref{prop: replicable_eta_quos} omits the theta quotient associated to the identity automorphism, i.e., the orbit type $1^N$. Since the theta series of even unimodular lattices are modular forms of weight $N/2$ for $\Gamma = \SL_2 \Z$, when $N \le 24$ there are limited possibilities. In particular, for $N = 8$ and $16$, we have $\dim M_{N/2}(\Gamma) = 1$ and thus the theta series are uniquely determined. When $N = 24$, we have $\dim M_{12}(\Gamma) = 2$, so that each lattice has a distinct theta series. In this case, however, the theta series is completely determined by the number of roots, i.e., vectors of norm $2$, and each such lattice has corresponding theta quotient $T_{1A}(q)$, albeit with varying constant term. This is no longer the case for higher rank lattices with $N \ge 32$; the dimension of the corresponding space of modular forms is at least $2$ \cite{Ono04}*{Prop.\ 1.25}.
\end{remark}

In Figures \ref{fig: lowdimetaquos}  and \ref{fig:etaquosdim24} we record those replicable functions from Proposition~\ref{prop: replicable_eta_quos} that arise as lattice theta quotients for 
the even unimodular code-lattices of rank $N \le 24$. The only such lattices with $N \le 16$ are $E_8$, $D_{16}^+$, and $E_8^2$. There are nine binary even self dual linear codes of length $24$. These give rise to nine code-lattices of rank $24$, corresponding to the nine Niemeier lattices containing an \emph{orthogonal $2$-frame} \cite{MH09}. These consist of the lattices labeled $N(D_{24}), N(D_{16}E_8), N(E_8^3), N(D_{12}^2), N(D_{10}E_7^2), N(D_8^3), N(D_6^4), N(D_4^6),$ and $N(A_1^{24})$. Of these nine codes, the Golay code $\mathcal{G}$ is the only one with minimum weight greater than $4$, and thus gives rise to both the Niemeier lattice $N(A_1^{24})$ via Construction $A$ and also to the Leech lattice via the construction in Section~\ref{subsec: supercodes}.

In Figure \ref{fig:etaquosdim24} the blank entries signify only that the computation was not performed; this choice was due to the large amount of computing time and memory required to determine the existence of the relevant subgroups of automorphisms of these lattices.

\begin{center}
{\bgroup
\def\arraystretch{1.5}
\begin{tabular}{|c | c || c|  c| c|}
\hline
Orbit Type & $\left( \frac{\theta_{L^G}(q)}{\eta_G(q)}\right)^{24/N}$ & $E_8$ & $D_{16}^+$ & $E_8^2$ \\ \hline
$1^N$ & $T_{1A}$ & Y & Y & Y \\ \hline
$2^{N/2}$ & $T_{4A}$ & Y & Y & Y \\ \hline
$4^{N/4}$ & $T_{8B}$ & Y & Y & Y \\ \hline
$8^{N/8}$ & $T_{16a}$ & N & Y & Y \\ \hline
$1^{N/4} \, 3^{N/4}$ & $T_{3A}$ & Y & Y & Y\\ \hline
$2^{N/8} \, 6^{N/8}$ & $T_{6b}$ & Y & N & Y\\ \hline
$2^{N/8} \, 6^{N/8}$ & $T_{12A}$ & N & Y & N \\ \hline
$1^{N/8} \, 7^{N/8}$ & $T_{7A}$ & Y & Y & Y\\ \hline
\end{tabular}
\egroup}
\captionof{figure}{theta quotients for code-lattices of rank $8$ and $16$} \label{fig: lowdimetaquos}
\end{center}

\begin{center}
{\bgroup
\def\arraystretch{1.5}
\begin{tabular}{|c|c||c|c|c|c|c|c|c|c|c|c|} \hline
Orbit Type & $\left( \frac{\theta_{L^G}(q)}{\eta_G(q)}\right)^{24/N}$ & $D_{12}^2$ & $D_{10} E_7^2$ & $D_8^3$ &  $D_6^4$ & $D_{24}$ & $D_4^6$ & $E_8^3$ & $D_{16}E_8$ & $A_1^{24}$ & Leech \\ \hline
$1^N$ & $T_{1A}$ & Y & Y & Y & Y & Y & Y & Y & Y & Y & Y\\ \hline
$2^{N/2}$ & $T_{4A}$ & Y & Y & Y & Y & Y & Y & Y & Y & Y & Y\\ \hline
$4^{N/4}$ & $T_{8B}$ & Y & -- &  Y & Y & Y & Y & Y & Y & Y & Y \\ \hline
$8^{N/8}$ & $T_{16a}$ & -- & -- &  -- & -- & -- & --  & -- & -- & N & N \\ \hline
$1^{N/4} \, 3^{N/4}$ & $T_{3A}$ & Y & Y & Y & Y & Y & Y & Y & Y & Y  & Y \\ \hline
$2^{N/8} \, 6^{N/8}$ & $T_{6b}$ & -- & -- & Y & -- & -- & Y & Y & -- & Y & Y\\ \hline
$2^{N/8} \, 6^{N/8}$ & $T_{12A}$ & -- & -- & -- & --  & -- & --& -- & -- & N & N\\ \hline
$1^{N/8} \, 7^{N/8}$ & $T_{7A}$ & -- & --  & -- & -- & -- & -- & Y & Y & Y & Y \\ \hline
\end{tabular}
\egroup}
\captionof{figure}{theta quotients for code-lattices of rank $24$}\label{fig:etaquosdim24}
\end{center}

\subsubsection{Additional replicable functions associated to the Leech lattice} \label{subsub: leechnewreplicable} As mentioned above, when $L$ is the Leech lattice $\Lambda_{24}$, there exist non-cyclic subgroups of $\Aut \mathcal{G}$ whose image in $\Aut L$ has orbit type $2^3 \, 6^3$ and the theta quotient for the corresponding fixed sublattice is $T_{6b}$. In Figure \ref{fig: replicablenoncyclicLeech}, we list some additional orbit types of (conjugacy classes of) non-cyclic subgroups of $\Aut \mathcal{G}$ whose images in $\Aut L$ have fixed sublattices with non-monstrous replicable theta quotients that do not arise from any cyclic subgroups. This gives a partial answer to Question~\ref{question: subgroups} for $\Lambda_{24}$ and addresses the statement at the end of Theorem~\ref{thm: mainintro}.

\begin{center}
\begin{tabular}{|c | l | c|}
\hline
Orbit type & Representative subgroup of $\Aut \mathcal{G}$ & Replicable function \\ \hline
\multirow{3}{*}{$2^3 \, 6^3$} & {\footnotesize $(1, 16)(2, 20)(3, 22)(4, 21)(5, 10)(6, 17)(7, 13)(8, 11)(9, 15)(12, 23)$} & \multirow{3}{*}{$T_{6b}$} \\
&  {\footnotesize$(14, 18)(19, 24),$} & \\
& {\footnotesize $(1, 12, 5)(3, 17, 15)(6, 22, 9)(7, 8, 14)(10, 23, 16)(11, 13, 18)$} & \\ 
 \hline
 \multirow{3}{*}{$3^2 \, 9^2$} & {\footnotesize $(1, 14, 11)(2, 16, 5)(3, 21, 22)(4, 12, 20)(6, 17, 19)(7, 10, 24)(8, 9,18)$} &\multirow{3}{*}{$T_{9b}$} \\
 & {\footnotesize$(13, 23, 15)$,} & \\
 & {\footnotesize $(1, 18, 13)(3, 4, 19)(6, 21, 12)(8, 23, 14)(9, 15, 11)(17, 22, 20)$} & \\ 
 \hline
  \multirow{5}{*}{$6^1 \, 18^1$} & {\footnotesize $ (1, 6)(2, 15)(3, 9)(4, 14)(5, 13)(7, 20)(8, 21)(10, 22)(11, 16)(12, 23)$} &\multirow{5}{*}{$T_{18b}$} \\
 & {\footnotesize $(17, 18)(19, 24)$,} & \\
 & {\footnotesize $(1, 18, 23)(2, 19, 21)(3, 5, 16)(4, 22, 20)(6, 12, 17)(7, 10, 14)(8, 24, 15)$} & \\ 
 & {\footnotesize $(9, 11, 13)$,} & \\
 & {\footnotesize $(1, 14, 13)(3, 20, 17)(4, 6, 5)(7, 9, 18)(10, 11, 23)(12, 16, 22)$} & \\ 
 \hline
 \multirow{3}{*}{$2^1 \, 22^1$} & {\footnotesize $  (1, 2)(3, 8)(4, 13)(5, 11)(6, 18)(7, 16)(9, 17)(10, 12)(14, 21)(15, 22)$} &\multirow{3}{*}{$T_{22a}$} \\
 & {\footnotesize$(19, 24)(20, 23)$,} & \\
 & {\footnotesize $ (1, 15, 13, 14, 18, 11, 10, 23, 9, 8, 7)(2, 16, 3, 17, 20, 12, 5, 6, 21, 4, 22)$} & \\ 
 \hline
 \multirow{5}{*}{$4^1 \, 20^1$} & {\footnotesize $   (1, 5)(2, 18)(3, 10)(4, 8)(6, 11)(7, 22)(9, 17)(12, 13)(14, 21)(15, 19)$} &\multirow{5}{*}{$T_{40a}$} \\
 & {\footnotesize $(16, 24)(20,23)$,} & \\
 & {\footnotesize $(1, 14)(2, 6)(3, 12)(4, 17)(5, 21)(7, 24)(8, 9)(10, 13)(11, 18)(15, 23)$} & \\ 
 & {\footnotesize $(16, 22)(19, 20)$,} & \\
 & {\tiny$(1, 5)(2, 10, 19, 8, 22, 18, 3, 15, 4, 7)(6, 13, 20, 9, 16, 11, 12, 23, 17, 24)(14, 21)$} & \\
 \hline
\end{tabular} 
\captionof{figure}{Replicable theta quotients associated to non-cyclic subgroups of automorphisms of the Leech lattice}
\label{fig: replicablenoncyclicLeech}
\end{center}

\begin{remark}
In fact, there is an automorphism of $\Lambda_{24}$ with  $2^3 6^3$ \emph{cycle type}, however it lies in $\Aut \Lambda_{24} \smallsetminus \Aut \mathcal{G}$. The corresponding lattice theta quotient gives the replicable function $T_{12A}$ \cite{L89}. Similarly, $T_{12A}$ also appears as a lattice theta quotient for the $E_8$ lattice, but for the conjugacy class of element with $2^1 6^1$ cycle type that lies in $\Aut E_8 \setminus \Aut \calH$ \cite{CLL92}. 
\end{remark}

\subsubsection{Higher rank examples}
Our methods and questions also apply to lattices of higher rank. There are 85 binary even self-dual linear codes $C$ of length 32 (see \cites{BilousVanRees, ConwayPless, ConwaySloanePless, Pless,  SloanePless} for the classification of binary self-dual codes of length at most 32.), each giving rise to a code-lattice $L$ via Construction A. Of these, only the three lattices isomorphic to  $E_8^4$, $E_8^2 \oplus D_{16}^+$, or $(D_{16}^+)^2$ have theta quotient $T_{1A}$ corresponding to the cycle type $1^{32}$, however many more have fixed sublattices with theta quotients $T_{4A}$ or $T_{8B}$, for example. The interested reader can run the \texttt{Magma} code accompanying the paper \cite{BK23} on any of the Type II linear codes $C$ in the database \cite{HM07} to recover additional information about the associated lattice theta quotients. Thus it is natural to ask for a classification of the possible replicable functions that can appear, especially for those higher rank lattices that are {orthogonally indecomposable}. 

\begin{Question} \label{question: higherdim} Can one give complete classification of the replicable functions that can arise as lattice theta quotients for even unimodular lattices in a given fixed rank $N$? Which replicable functions never appear in rank $N \ge 32$ for those lattices that are not direct sums of lattices of dimension at most 24?
\end{Question} 

\section{Code theoretic characterizations of fixed sublattices} \label{sec: characterizationcodes}
In Section~\ref{sec: repfunctionscodelats}, we proved Proposition \ref{prop: replicable_eta_quos} in which we enumerated pairs of automorphisms with a given cycle type and fixed sublattices of a code-lattice $L$ that give rise to replicable lattice theta quotients. In this section, we prove that for several such cycle types, this information can be recovered purely via properties of subcodes of the code $C$ from which the code-lattices are constructed. We focus on those subcodes that exist for the unique doubly even self-dual code of length $8$ and provide data for codes of length at most $32$, noting that there are other characterizations of subcodes giving rise to the same replicable lattice theta quotients for length larger than $8$.

\subsection{Identities involving Jacobi theta functions} We record several identities regarding the Jacobi theta functions defined in \eqref{jacobithetas} which will be needed in the next several sections. 

\begin{Lemma} \label{lem: jacobidentities} Let $N$ be a positive integer. The following identities hold.
\begin{enumerate}[itemsep = .3 em]
     \item \label{part: T4} $\vartheta_2(q)^2 = 2\vartheta_2(q^2)\vartheta_3(q^2)$ 
    \item \label{part: T5} $ \vartheta_3(q)^2 = \vartheta_3(q^2)^2 + \vartheta_2(q^2)^2$ 
    \item \label{part: thetaA} $\theta_{A_1^N}(q) = \vartheta_3(q)^N$, where $A_1^N \cong \Z^N$ denotes the standard lattice of rank $N$.
    \item \label{part: thetaD} $\theta_{D_{N}^\ast}(q) = \vartheta_3(q)^{N} + \vartheta_2(q)^{N}$
    \item \label{part: etaquo} 
    $ \dfrac{\eta(q)^2}{\eta(q^2)^2} = \dfrac{\vartheta_4(q^2)}{\eta(q^2)}$
\end{enumerate}
\end{Lemma}

\begin{proof}
    We have \eqref{part: T4} is \cite{KT86}*{(T4)}, \eqref{part: T5} is \cite{KT86}*{(T5)}, \eqref{part: thetaA} is \cite{CS99}*{\S4.5}, \eqref{part: thetaD} is \cite{CS99}*{\S 4.7(96)}, and \eqref{part: etaquo} follows from the identity $\vartheta_4(q^2) = \eta(q)^2/\eta(q^2)$ \cite{KT86}*{(T22)}.
\end{proof}

\begin{Lemma} \label{lem: matchingcoeffsAD}
The theta series $\theta_{A_1^{N/2}}(q^2)$ and $\theta_{D_{N/2}^\ast}(q^2)$ have the same coefficients on the even (resp. odd) powers of $q$ when $N \equiv 8 \bmod{16}$ (resp. $N \equiv 0 \bmod{16}$). 
\end{Lemma}

\begin{proof}
        By Lemma~\ref{lem: jacobidentities}\eqref{part: thetaA} and \eqref{part: thetaD}, we need only consider the even (resp. odd) coefficients of $\vartheta_2(q^2)^{N/2}$. For a fixed integer $k$, let $a_k$ denote the coefficient on $q^k$. Then by (\ref{jacobithetas}) \[ a_k = \#\left\{ (y_1, \dots, y_{N/2}) \in \Z^{N/2} : \sum_{i=1}^{N/2} y_i^2 + \sum_{i=1}^{N/2} y_i + \frac{N}{8} = k \right\}.\] Since $\sum y_i^2$ and $\sum y_i$ have the same parity, we have $a_k = 0$ when $N \equiv 8 \bmod{16}$ and $k$ is even (resp.\ $N \equiv 0 \bmod{16}$ and $k$ is odd).
\end{proof}

\subsection{Fixed subcodes giving rise to prescribed lattice theta series} Given an automorphism $\gbar \in \Aut C$, we now give conditions for the fixed subcode $C^\gbar$ which guarantee that the fixed sublattice of the corresponding code-lattice has theta series that agrees with that of the lattice $A_1(r)^{N/r}$. This in turn guarantees a replicable lattice theta quotient.

\begin{Theorem*}[Theorem~\ref{thm: rA1^N/r}]
    Let $C$ be a doubly even self dual linear code of length $N$, let $r | N$, and let $\gbar \in \Aut C$ have cycle type $r^{N/r}$. Suppose that the fixed subcode $C^\gbar$ has dimension $\dim C^{\gbar} = \frac{1}{r} \dim C$ spanned by $\{B_1, \dots, B_{N/2r}\}$. Further suppose $B_i \cap B_j = \emptyset$ for all $i \ne j$ and $\cup_i B_i = \{1, \dots, N\}$. Then the corresponding sublattice fixed by $g = \iota(\gbar)$ has theta series given by $\theta_{L^g}(q) = \vartheta_3(q^r)^{N/r}$. In particular the theta series of the fixed sublattice is the same as that of the lattice $A_1(r)^{N/r}$.
\end{Theorem*} 

\begin{proof} 
We will show that $\theta_{L^g}(q) = (\vartheta_2(q^{2r})^2 + \vartheta_3(q^{2r})^2)^{N/2r}$, which is the same as $\theta_{L^g}(q) = \vartheta_3(q^r)^{N/r}$ by Lemma~\ref{lem: jacobidentities}\eqref{part: T5}. In fact, we will prove the equivalent statement that
\begin{equation} \label{thetabinom} \theta_{L^g}(q) = \sum_{k=0}^{N/2r} \binom{N/2r}{k} \, \vartheta_2(q^{2r})^{2k} \, \vartheta_3(q^{2r})^{2(N/2r - k)}. \end{equation}

By equation \eqref{fixedsublattice}, in order to find the fixed sublattice $L^g$ we can take the union over each codeword in the fixed subcode $C^{\gbar}$ as follows. Write \begin{equation} L^g=\bigcup_{B\in C^{\gbar}}\left\{\textstyle \frac{1}{2} \alpha_B + \sum_i x_i \alpha_i:x_i=x_j
\mbox{ if $i$ and $j$ are in the same cycle of $\gbar$} \right \}. \end{equation} Since $B_i \cap B_j = \emptyset$ for all $i \ne j$ with $\cup_i B_i = \{1, \dots, N\}$, and \[ \dim C^{\gbar} = \frac{1}{r} \dim C = \frac{1}{r} \cdot \frac{N}{2} =  \frac{N}{2r},\] we have $|B_i| = 2r$ for each $i$. Any codeword fixed by $\gbar$ has the form $B = \sum_{j \in \mathcal{J}} B_j$ for some (possibly empty) set of indices $\mathcal{J} \subset \{1, \dots, N/{2r} \}$. We let $k \coloneqq |\mathcal{J}|$ so that the weight of such a codeword $B$ is $2kr$. Since $\gbar$ has cycle type $r^{N/r}$, it has the disjoint cycle decomposition $\Omega = \cup_{i=1}^{N/r} O_i$
with $|O_i| = r$ for all $i$. By Lemma~\ref{lem: orbitdecomp}, after possibly reordering, we can write $B = \cup_{i=1}^{2k} O_i$ and $\Omega \smallsetminus B = \cup_{i=2k+1}^{N/r} O_i$. Thus the corresponding discrete subset in the union defining $L^g$ is given by
\[\bigcup_{i=1}^{2k}\left(\Z + \textstyle \frac{1}{2}\right)\alpha_{O_i} \cup \bigcup_{i=2k+1}^{N/r} \Z \, \alpha_{O_i},\]
which therefore has theta series
\begin{equation} \label{eqn: thetaserA1Nr} \vartheta_2(q^{2r})^{2k} \vartheta_3(q^{2r})^{N/r - 2k} = \vartheta_2(q^{2r})^{2k} \vartheta_3(q^{2r})^{2(N/2r - k)},\end{equation}
by the definitions of the Jacobi theta functions $\vartheta_3$ and $\vartheta_2$ in \eqref{jacobithetas}. Finally, the number of such codewords in $C^{\gbar}$ with $|B| =2kr$ is given by $\binom{N/2r}{k}$. Together with the theta series calculation \eqref{eqn: thetaserA1Nr}, this recovers formula \eqref{thetabinom}.
\end{proof}

\begin{Corollary}
     Assume the code $C$ and $\gbar \in \Aut C$ satisfy the hypothesis of Theorem~\ref{thm: rA1^N/r}. Then the lattice theta quotient corresponding to $g=\iota(\bar{g})$ is replicable. In particular, for $r=1,2,4,8$ then $(\theta_{L^g}(q)/\eta_g(q))^{24/N}$ is equal to $T_{1A}(q)$, $T_{4A}(q)$, $T_{8b}(q)$, and $T_{16a}(q)$ respectively.
\end{Corollary}

\begin{proof}
This follows immediately from Proposition~\ref{prop: replicable_eta_quos} and Theorem~\ref{thm: rA1^N/r}. 
\end{proof}

When $\gbar$ has cycle type $2^{N/2}$, the following conditions guarantee a sublattice which shares a theta series with the lattice $D_{N/2}^\ast(2)$.

\begin{Proposition} \label{prop:2DN/2*}Let $C$ be a doubly even self-dual linear code of length $N$ and let $\gbar \in \Aut C$ have cycle type $2^{N/2}$. Suppose the fixed subcode $C^\gbar$ has dimension $\dim C^{\gbar} = \frac{1}{2} \dim C +1 = \frac{N}{4}+1$  and is spanned by $\{B_0, B_1, \dots, B_{N/4}\}$. Further suppose that
\begin{enumerate}
    \item $|B_0| = \frac{N}{2}$,
    \item $|B_j| = 4$ for $1 \le j \le \frac{N}{4}$,
    \item $B_i \cap B_j = \emptyset$ for $1 \le i,j \le \frac{N}{4}$ with $i \ne j$, and
    \item $|B_0 \cap B_j| = 2$ for $1 \le j \le \frac{N}{4}$.
\end{enumerate}
Then the sublattice fixed by $g = \iota(\gbar)$ has theta series $\theta_{L^g}(q) = \vartheta_3(q^2)^{N/2} + \vartheta_2(q^2)^{N/2}$. In particular, the theta series of the fixed sublattice agrees with that of the lattice $D_{N/2}^\ast(2)$.
\end{Proposition}

\begin{proof} We will show that
\begin{align} \label{eqn: 2DN/2*} \theta_{L^g}(q) &= \vartheta_3(q^2)^{N/2} +  2^{N/4}(\vartheta_2(q^4)^2\vartheta_3(q^4)^2)^{N/4} \end{align}
which is equivalent to the desired form by Lemma~\ref{lem: jacobidentities}\eqref{part: T4}.

Let $S$ be the subcode of $C^{\gbar}$ spanned by $\{B_1, \dots, B_{N/4} \}$. Then $S$ is isomorphic to the code in Theorem~\ref{thm: rA1^N/r} with $r = 2$ . Therefore the sum over $B \in S$ of the theta series of the discrete subset associated to such $B$ is given by $\vartheta_3(q^2)^{N/2}$. Thus it remains to compute the theta series of the discrete subsets associated to the codewords $B \in C^{\gbar} \smallsetminus S$. Any such codeword has the form $B \coloneqq B_0 + \sum_{j \in \mathcal{J}} B_j$ for some (possibly empty) set of indices $\mathcal{J} \subset \{1, \dots, \frac{N}{4} \}$. Recall from \eqref{eqn: codewordaddition} that the sum of any two codewords is defined by $B + B' \coloneqq B \cup B' \smallsetminus B \cap B'$. Thus \[ |B| = \textstyle |B_0| + \big|\sum_{j \in \mathcal{J}} B_j\big| - \, 2\big|B_0 \cap \sum_{j \in \mathcal{J}} B_j \big| =  \frac{N}{2} + 4|\mathcal{J}| - 2(2|\mathcal{J}|) = \frac{N}{2}.\] Hence in total there are $2^{N/4}$ such codewords of weight $\frac{N}{2}$, each of which corresponds to a discrete subset with theta series given by $(\vartheta_3(q^4)^2 \vartheta_2(q^4)^2)^{N/4}$. Taking the sum of these terms together with the sum for those $B \in S$ gives the equation \eqref{eqn: 2DN/2*}.
\end{proof}

In Figure~\ref{fig: codedata} we record the number of doubly even binary self-dual codes of length $N \le 32$ that contain an element $\gbar \in \Aut C$ of cycle type $2^{N/2}$ fixing a subcode $C^{\gbar}$ of the form given in Theorem~\ref{thm: rA1^N/r} (with $r=2$) or Proposition~\ref{prop:2DN/2*}, denoted by $C_A$ and $C_D$, respectively.

\begin{center}
\def\arraystretch{1.3}
    \begin{tabular}{|c||c|c|c|}
    \hline 
    $N$ &  $\#\left\{C :\Large \substack{C^\gbar \cong C_A \, \text{for some} \\  \, \gbar \, \in \Aut C \qquad}\right\}$ & $\#\left\{C :\Large \substack{C^\gbar \cong C_D \, \text{for some} \\ \gbar \, \in \Aut C \qquad}\right\}$ & $\#\{C : \text{len}(C) = N \}$\\
     \hline 
    8 & 1 & 1 & 1 \\ \hline 
16 & 1 & 2 & 2 \\  \hline
24 & 6 & 5 & 9\\ \hline
32 & 19 & 12 & 85\\ \hline
\end{tabular}
\captionof{figure}{Codes of length $\le 32$ containing fixed subcodes of the form in Theorem~\ref{thm: rA1^N/r} or Proposition~\ref{prop:2DN/2*}}
\label{fig: codedata}
\end{center}

\section{Relationships between coefficients of characters of fixed subVOAs} \label{sec: coeffsVOAchars}

Throughout this section, we assume that $L$ is a code-lattice constructed from a doubly even self dual code $C$ of length $N$ (necessarily divisible by $8$) via either Construction A or the lattice construction for super codes described in Section~\ref{subsec: supercodes}. We prove results about characters of the subVOAs of the lattice-VOA $V_L$ fixed by lifts of lattice automorphisms in the image of $\Aut C$ under the natural embedding. In order to distinguish between conjugacy classes of automorphisms with the same cycle type, we shall decorate the cycle type with either $rep$ or $nr$ to denote replicable or non-replicable associated trace function, respectively. Although there may be several conjugacy classes with the same associated trace function, we will choose just one such class which we assume to be fixed throughout. 

\subsection{Characters of fixed subVOAs for certain cycle types} \label{subsec: charsfixedsubVOAs} We consider automorphisms of codes with cycle type $2^{N/2}$. In particular, we consider the case when there are at least two such automorphisms, where one fixes a sublattice isomorphic to $A_1(2)^{N/2}$ and the other fixes a sublattice isomorphic to $D_{N/2}^\ast(2)$. We observe properties of the characters of their fixed subVOAs $\Ch V^G$, where $G$ is the cyclic subgroup generated by an automorphism either of cycle type $2^{N/2}_{rep}$ or $2^{N/2}_{nr}$. Given a lattice $L$ and the corresponding lattice-VOA $V_L$, we let $V_L^+$ denote the subVOA fixed by the automorphism $-\text{id} \in \Aut L$.

We find that the characters corresponding to a $2^{N/2}_{nr}$ cycle type have coefficients that coincide with the character of $V_L^+$ while the characters corresponding to a $2^{N/2}_{rep}$ cycle type have coefficients that coincide with the character of $V_{L_0}^+$, where $L_0$ denotes the kernel for order doubling. After gathering these results, we end the subsection by proving Theorem~\ref{thm: etaquoandcharidentities} relating the theta quotients to the fixed subVOA characters.

Figure~\ref{fig: charVG} below records the characters of the  subVOAs fixed by lifts of automorphisms with cycle types $2^{N/2}_{rep}$ and $2^{N/2}_{nr}$ of the codes whose corresponding code-lattices are $E_8$, $D_{16}^+$, and $N(D_{12}^2)$.

\begin{center}
\renewcommand{\arraystretch}{1.2}
\begin{tabular}{|c|c|c|l|}
\hline
Lattice & Cycle type & Sublattice & $\Ch V^G$ \\ \hline
\multirow{4}{*}{$E_8$} & \multirow{2}{*}{$2^4_{rep}$} & \multirow{2}{*}{$A_1(2)^4$} &  $q^{-\frac{1}{3}}(1 + 64q + 
    {1052} q^2 + 8704 q^3 + 
{53382}q^4 + 264448q^5$ \\ &&& $ + \, {1133112}q^6 \dots )$ \\ \cline{2-4}
    {} & \multirow{2}{*}{$2^4_{nr}$} & \multirow{2}{*}{$D_4^\ast(2)$} & $q^{-\frac{1}{3}}(1 + 136q + 
    {2076}q^2 + 17472q^3 + 
    {106630}q^4$ \\ & & & $+ \,  529184q^5 +
    {2265656}q^6 \dots )$ \\ \hline \hline
    \multirow{4}{*}{$D_{16}^+$} &  \multirow{2}{*}{$2^8_{rep}$} & \multirow{2}{*}{$2 A_1^8$} &$q^{-\frac{2}{3}}(1 + 256q + 18552q^2 + 533504q^3 +
    8685596q^4$ \\ & & & $ + \, 98667008q^5 + 874939328q^6  \dots )$\\ \cline{2-4}
    {} & \multirow{2}{*}{$2^8_{nr}$} & \multirow{2}{*}{$D_8^\ast(2)$} & $q^{-\frac{2}{3}}(1 + {256}q + 35064q^2 + {1057792}q^3 +
    17338396q^4$ \\ & & & $+\, {197233152}q^5 + 1749600192q^6\dots)$ \\ \hline \hline
    \multirow{4}{*}{$N(D_{12}^2)$} & \multirow{2}{*}{$2^{12}_{rep}$} & \multirow{2}{*}{$A_1(2)^{12}$} & $q^{-1}(1 + 224 q + {57620}q^2 + 5570560q^3 + {218540994}q^4$\\ & & & $+ \, 5082972160q^5 +  {83449286360}q^6 + \dots)$ \\ \cline{2-4}
    {} & \multirow{2}{*}{$2^{12}_{nr}$} & \multirow{2}{*}{$D_{12}^\ast(2)$} & $q^{-1}(1 + 288q + {98580}q^2 + 10749952q^3 + {432155586}q^4$ \\ & & &  $+\, 10123001856q^5+ {166601412312}q^6 +\dots)$ \\ \hline
\end{tabular}
\captionof{figure}{Characters of subVOAs fixed by lifts of automorphisms with $2^{N/2}$ cycle type}
\label{fig: charVG}\end{center}

We compare the coefficients of the characters of the fixed subVOAs of the $E_8$ lattice-VOA to the characters of $V_{E_8}^+$ and $V_{D_8}^+$. In this case, by Lemma \ref{lem: kerorddoubE8}, the kernel for order doubling for the $2^4_{rep}$ cycle type  is $D_8$. We note that the coefficients on the even powers of $q$ are equal.
\begin{align*} 
\Ch V_{D_8}^{+} &= q^{-\frac{1}{3}}(1 + 56q + \textbf{1052}q^2 + 8640q^3 + \textbf{53382}q^4 + 264160q^5 + \textbf{1133112}q^6 + \dots)\\
\Ch V_{E_8}^{+} & =q^{-\frac{1}{3}}(1 + 120q + \textbf{2076}q^2 + 17344q^3 + \textbf{106630}q^4 + 528608q^5 + \textbf{2265656}q^6 + \dots)
\end{align*}

Similarly, taking $L$ to be the Niemeier lattice $N(D_{12}^2)$ and $V_L$ the corresponding lattice-VOA, we compare the coefficients of the characters of the fixed subVOAs to those of the characters of $V_L^+$ and $V_{L_0}^+$ where $L_0$ is the kernel for order doubling for the $2^{12}_{rep}$ cycle type. Once again, we observe the matching coefficients on the even powers of $q$.
\begin{align*} 
\Ch V_{L_0}^{+} &= q^{-1}(1 + 200q + \textbf{57620}q^2 + 5568512q^3 + \textbf{218540994}q^4 +
    5082923008q^5 \\ & \quad + \textbf{83449286360}q^6 + \dots) \\
\Ch V_{N(D_{12}^2)}^{+} & = q^{-1}(1 + 264q + \textbf{98580}q^2 + 10745856q^3 + \textbf{432155586}q^4 +10122903552q^5\\ &\quad +  \textbf{166601412312}q^6+ \dots)
\end{align*}
The same phenomena occur for at least two other Niemeier lattices, namely $N(D_8^3)$ and $N(D_4^6)$.

\begin{Proposition}\label{prop: coeffmatchingchars} 
Assume there exists $g \in \Aut L$ of cycle type $2^{N/2}$ with associated fixed sublattice isometric to $D_{N/2}^\ast(2)$, and let $G \subset \Aut V_L$ be the subgroup generated by the standard lift $\hat{g}$ of $g$. Then the coefficients on the even powers of $q$ in the character of the fixed subVOA $\Ch V^G$ agree with those of the character $\Ch V_L^+$.
\end{Proposition}

\begin{proof}
Consider $-\text{id} \in \Aut L$ and its standard lift to $\Aut V_L$. This automorphism has cycle type $1^{-N} 2^N$ with trivial fixed sublattice. Hence the character of $V_L^+$ is
\begin{equation}
    \Ch V_L^{+} = \frac{1}{2}\left(\frac{\theta_L(q)}{\eta(q)^N} + \frac{\eta(q)^N}{\eta(q^2)^N}\right),
\end{equation}
while the character $\Ch V^G$ is given by 
\begin{equation} \label{eqn:eta2}
    \Ch V^G = \frac{1}{2}\left(\frac{\theta_L(q)}{\eta(q)^N} + \frac{\theta_{D_{N/2}^\ast}(q^2)}{\eta(q^2)^{N/2}}\right).
\end{equation}

By Lemma~\ref{lem: matchingcoeffsAD}, the coefficients on the even powers of $q$ in the second term of \eqref{eqn:eta2} are equal to the coefficients on the even powers of $q$ in the theta quotient \[ \frac{\theta_{A_1^{N/2}}(q^2)}{\eta(q^2)^{N/2}} = \frac{\vartheta_3(q^2)^{N/2}}{\eta(q^2)^{N/2}} = \left(\frac{\vartheta_3(q^2)}{\eta(q^2)}\right)^{N/2}. \] 
By Lemma~\ref{lem: jacobidentities}\eqref{part: etaquo}, we have that 
\[ \frac{\eta(q)^N}{\eta(q^2)^N} = \left(\frac{\eta(q)^2}{\eta(q^2)^2}\right)^{N/2} = \left(\frac{\vartheta_4(q^2)}{\eta(q^2)}\right)^{N/2}.\]

Since $\vartheta_3(q^2) = \sum q^{j^2}$ and $\vartheta_4(q^2) = \sum (-1)^j q^{j^2}$ we therefore have agreement on the coefficients on the even powers of $q$, as desired.
\end{proof}

\begin{Proposition} \label{prop: coeffmatchingorderdoubled}
Suppose that $N \equiv 8 \bmod{16}$ and that there exists $g \in \Aut L$ of cycle type $2^{N/2}$ with associated fixed sublattice isometric to $A_1(2)^{N/2}$. Let $G \subset \Aut V_L$ be the subgroup generated by the standard lift of $g$ which has order doubling and kernel for order doubling denoted $L_0$. Then the coefficients on the even powers of $q$ in the character of the fixed subVOA $\Ch V_L^G$ agree with those of $\Ch V_{L_0}^+$.
\end{Proposition}

\begin{proof}
    We must compare
    \begin{eqnarray*}
        \Ch V_L^G  & = & \frac{1}{4} \left( \frac{\theta_{A_1^{N/2}}(q^2)}{\eta(q^2)^{N/2}} +   \frac{2\theta_{L_0}(q) - \theta_L(q)}{\eta(q)^N} +  \frac{\theta_{A_1^{N/2}}(q^2)}{\eta(q^2)^{N/2}} +  \frac{\theta_L(q)}{\eta(q)^N} \right) \\ 
        & = & \frac{1}{2} \left( \frac{\theta_{A_1^{N/2}}(q^2)}{\eta(q^2)^{N/2}} +   \frac{\theta_{L_0}(q)}{\eta(q)^N}  \right) \\ 
    \end{eqnarray*}
    with the character
    \[ \Ch V_{L_0}^{+}  = \frac{1}{2} \left( \frac{\eta(q)^{N/2}}{\eta(q^2)^{N/2}} +   \frac{\theta_{L_0}(q)}{\eta(q)^N}  \right), \]
    where the element $-\text{id} \in \Aut L_0 \subset \Aut L$ has cycle type $1^{-N} 2^N$ with trivial fixed sublattice. The result then follows from the argument at the end of Proposition~\ref{prop: coeffmatchingchars}.
\end{proof}

We also consider analogous results for the lattice theta quotients associated to these two cycle types. We prove that their theta quotients share coefficients, both with each other and with characters of other known lattice-VOAs. Figure~\ref{fig: etaquos} below records the theta quotients of the fixed sublattices by automorphisms with cycle types $2^{N/2}_{rep}$ and $2^{N/2}_{nr}$ for several low rank code-lattices. \smallskip
\begin{center}
\renewcommand{\arraystretch}{1.25}
\begin{tabular}{|c|c|c|l|}
\hline
Lattice & Cycle type & Sublattice & Theta quotient \\ \hline
\multirow{4}{*}{$E_8$} & \multirow{2}{*}{$2^4_{rep}$} & \multirow{2}{*}{$A_1(2)^4$} & $q^{-\frac{1}{3}}(1 + 8q + 
    \textbf{28}q^2 + 64q^3 + 
    \textbf{134}q^4 + 288q^5$\\ &&& $+\, 
    \textbf{568}q^6+ \dots)$ \\ \cline{2-4}
    \multirow{2}{*}{} & \multirow{2}{*}{$2^4_{nr}$} & \multirow{2}{*}{$D_4^\ast(2)$} & $q^{-\frac{1}{3}}(1 + 24q + 
    \textbf{28}q^2 + 192q^3 + 
    \textbf{134}q^4 + 864q^5$\\ &&& $+\, 
    \textbf{568}q^6+ \dots)$ \\ \hline \hline
    \multirow{4}{*}{$D_{16}^+$} & \multirow{2}{*}{$2^8_{rep}$} & \multirow{2}{*}{$A_1(2)^8$} &$q^{-\frac{2}{3}}(1+\textbf{16}q + {120}q^2 + \textbf{576}q^3 +
    {2076}q^4 + \textbf{6304}q^5$ \\ &&& $ \, + {17344}q^6+ \dots )$ \\ \cline{2-4} 
    \multirow{2}{*}{} & \multirow{2}{*}{$2^8_{nr}$} & \multirow{2}{*}{$D_8^\ast(2)$} & $q^{-\frac{2}{3}}(1+\textbf{16}q + 376q^2 + \textbf{576}q^3 +
    6172q^4 + \textbf{6304}q^5$ \\ &&& $+\, 52160q^6 + \dots)$ \\ \hline \hline
    \multirow{4}{*}{$N(D_{12}^2)$} & \multirow{2}{*}{$2^{12}_{rep}$} & \multirow{2}{*}{$A_1(2)^{12}$} & $q^{-1}(1 + 24q + \textbf{276}q^2 + 2048q^3 + \textbf{11202}q^4 +
    49152q^5$ \\ & & & $+ \, \textbf{184024}q^6 +  614400q^7 + \textbf{1881471}q^8 +\dots)$ \\ \cline{2-4}
    {} & \multirow{2}{*}{$2^{12}_{nr}$} & \multirow{2}{*}{$D_{12}^\ast(2)$} & $q^{-1}(1 + 24q + \textbf{276}q^2 + 6144q^3 + \textbf{11202}q^4 +
    147456q^5$ \\ & & & $+\, \textbf{184024}q^6 + 1843200q^7 + \textbf{1881471}q^8 +
    \dots)$\\ \hline
\end{tabular} \captionof{figure}{Theta quotients of fixed sublattices by automorphisms with $2^{N/2}$ cycle type}
\label{fig: etaquos}\end{center}
\clearpage
\begin{Proposition} \label{prop: matchingcoeffs}  \,
\begin{enumerate}[itemsep = .3 em]
    \item \label{part: matching} The theta quotients for the fixed sublattices associated to elements with cycle types $2^{N/2}_{rep}$ and $2^{N/2}_{nr}$ in $\Aut L$ have the same coefficients on even (resp. odd) powers of $q$ when $N \equiv 8 \bmod{16}$ (resp. $0 \bmod{16}$).
    \item \label{part: evenmatching} The character $\Ch V_{D_{N/2}}(q^2)$ of the $D_{N/2}(2)$ lattice-VOA and the theta quotients associated to the cycle types $2^{N/2}_{rep}$ and $2^{N/2}_{nr}$ have the same coefficients on even powers of $q$ for $N \equiv 8 \bmod{16}$.
\end{enumerate}
\end{Proposition}

\begin{proof}
    \eqref{part: matching} The fixed sublattices are both associated to a conjugacy class of element in $\Aut L$ with cycle type $2^{N/2}$. Hence the two theta quotients have the same denominator, namely $\eta(q^2)^{N/2}$. Thus it suffices to compare the coefficients of the corresponding theta series, i.e., $\theta_{A_1^{N/2}}(q^2)$ and $\theta_{D_{N/2}^\ast}(q^2)$. This is the content of Lemma~\ref{lem: matchingcoeffsAD}. 
    
    \noindent \eqref{part: evenmatching} Note that the character of the $D_{N/2}(2)$ lattice-VOA has denominator $\eta(q^2)^{N/2}$, thus again we need only compare the theta series of this lattice to the theta series of $A_1(2)^{N/2}$. Let $a_k$ denote the coefficient on $q^k$ in the theta series of the lattice $D_{N/2}(2)$. Then \[a_k = \# \left\{ (x_1, \dots, x_{N/2}) \in \Z^{N/2} : \sum_i x_i \equiv 0 \bmod{2} \text{ and } \sum x_i^2 = k\right\}. \] Since $\sum x_i^2$ and $\sum x_i$, have the same parity, $a_k$ is nonzero if and only if $k$ is even. We see from the proof of \eqref{part: matching} that these coefficients agree with the corresponding coefficients in the theta series of $A_1(2)^{N/2}$ and $D_{N/2}^\ast(2)$, as desired.
\end{proof}

\begin{remark} The Golay code $\calG$ has exactly one conjugacy class of automorphism with cycle type $2^{12}$. For both $N(A_1^{24})$ and $\Lambda_{24}$ the corresponding theta quotient is the replicable function $T_{4A}(q)$. In the case of $\Lambda_{24}$, the fixed sublattice is isometric to the lattice $2D_{12}^+$ and the standard lift of the lattice automorphism to the Leech lattice-VOA has order doubling. The kernel for order doubling is an index $2$ sublattice of $\Lambda_{24}$ with theta series \[\theta_{L_0}(q) = 1 + 98256q^2 + 8384512q^3 + 199066704q^4 + 2314125312q^5 +O(q^6).\]
\end{remark}

\begin{Theorem*}[Theorem~\ref{thm: etaquoandcharidentities}]
Suppose there exists an automorphism $g_1 \in \Aut L$ with cycle type $2^{N/2}$ and fixed sublattice isometric to $A_1(2)^{N/2}$. Let $L_0$ denote the kernel for order doubling associated to the lift of $g_1$ to the VOA automorphism $\hat{g_1} \in \Aut V_L$. Then 
\begin{equation}\label{eqn: repetaidentity}
    \frac{\theta_{L^{g_1}}(q)}{\eta(q^2)^{N/2}} = \Ch V_L^{\hat{g_1}}(q) - \Ch V_{L_0}^{+}(q) + \Ch V_{D_{N/2}}(q^2).   
\end{equation}
Moreover, if there exists an automorphism $g_2 \in \Aut L$ with cycle type $2^{N/2}$ and fixed sublattice isometric to $D_{N/2}^\ast(2)$, then the characters of the fixed point subVOAs are related via
 \begin{equation} \label{eqn: nretaidentity}
     \frac{\theta_{L^{g_2}}(q)}{\eta(q^2)^{N/2}}  = 2(\Ch V_L^{\hat{g_2}}(q) - \Ch V_L^{+}(q) ) - \bigl(\Ch V_L^{\hat{g_1}}(q) - \Ch V_{L_0}^{+}(q)\bigr) + \Ch V_{D_{N/2}}(q^2).    
\end{equation}  
\end{Theorem*} 

\begin{proof} By Lemma~\ref{lem: jacobidentities}\eqref{part: etaquo}, we have 
\[ \frac{\eta(q)^N}{\eta(q^2)^N} = \left(\frac{\vartheta_4(q^2)}{\eta(q^2)}\right)^{N/2}.\]

Then, using the definitions of the theta series for $A_1(2)^{N/2}$ and $D_{N/2}^\ast(2)$ given in Lemma~\ref{lem: jacobidentities}\eqref{part: thetaA} and \eqref{part: thetaD}, respectively, the right hand side of \eqref{eqn: nretaidentity} is 
\begin{align*}
    & = 2 \cdot \frac{1}{2} \left( \frac{\vartheta_3(q^2)^{N/2} + \vartheta_2(q^2)^{N/2}}{\eta(q^2)^{N/2}} + \frac{\theta_L(q)}{\eta(q)^N}\right) 
    - 2 \cdot \frac{1}{2} \left( \frac{\theta_L(q)}{\eta(q)^N} + \frac{\vartheta_4(q^2)^{N/2}}{\eta(q^2)^{N/2}} \right) \\
    & \quad - \frac{1}{2}\left( \frac{\vartheta_3(q^2)^{N/2}}{\eta(q^2)^{N/2}} + \frac{\theta_{L_0}(q)}{\eta(q)^N}\right) +
     \frac{1}{2} \left(\frac{\theta_{L_0}(q)}{\eta(q)^N} + \frac{\vartheta_4(q^2)^{N/2}}{\eta(q^2)^{N/2}}\right) + \frac{1}{2}\left( \frac{\vartheta_3(q^2)^{N/2} + \vartheta_4(q^2)^{N/2}}{\eta(q^2)^{N/2}} \right) \\
    & = \frac{\vartheta_3(q^2)^{N/2} + \vartheta_2(q^2)^{N/2}}{\eta(q^2)^{N/2}} = \frac{\theta_{L^{g_2}}(q)}{\eta(q^2)^{N/2}}.
\end{align*}
Similarly, the right hand side of equation \eqref{eqn: repetaidentity} is 
\begin{align*}
   & = \frac{1}{2}\left( \frac{\vartheta_3(q^2)^{N/2}}{\eta(q^2)^{N/2}} + \frac{\theta_{L_0}(q)}{\eta(q)^N}\right)  - \frac{1}{2} \left(\frac{\theta_{L_0}(q)}{\eta(q)^N} + \frac{\vartheta_4(q^2)^{N/2}}{\eta(q^2)^{N/2}}\right) + \frac{1}{2}\left( \frac{\vartheta_3(q^2)^{N/2} + \vartheta_4(q^2)^{N/2}}{\eta(q^2)^{N/2}} \right) \\
   & = \frac{\vartheta_3(q^2)^{N/2}}{\eta(q^2)^{N/2}} = \frac{\theta_{L^{g_1}}(q)}{\eta(q^2)^{N/2}}.
\end{align*}
\end{proof}
The identity \eqref{eqn: repetaidentity} also holds for the Leech lattice, up to a constant. To the authors' knowledge, a statement analogous to Theorem~\ref{thm: etaquoandcharidentities} does not easily generalize to other cycle types. 

\subsection{Code theoretic characterization of order doubling}
Certain properties of the code-lattices and lattice-VOAs from Section~\ref{subsec: charsfixedsubVOAs} are detectable by the automorphism groups and fixed subcodes of the linear code. Let $C$ be a doubly even self-dual linear code of length $N$, and assume that $L$ is the code-lattice of $C$ as constructed in Proposition \ref{prop:tasaka}. That is, $L$ is the lattice generated by $\{ \alpha_i, \frac{1}{2}\alpha_B : 1 \le i \le N, B \in C\}$, where $\{\alpha_i\}_{i=1}^N$ is a basis for $\R^N$ such that $\langle \alpha_i, \alpha_j \rangle = 2 \delta_{i,j}$. We again view codewords $B$ in the lens of Section~\ref{subsec: equivconst}, that is, as a subset $B \subset \{1, \dots, N\}$ encoding the indices of $B$ for which the coordinate is $1$, with addition defined by equation \eqref{eqn: codewordaddition} as the symmetric difference of sets.

The following theorem gives a characterization, in terms of a property of the code $C$, of when an even order lattice automorphism $g$ lifts to an automorphism $\hat{g}$ of the lattice-VOA $V_L$ that has order doubling. 

\begin{Theorem} \label{thm: orddoubcode}Suppose $\gbar \in \Aut C$ has even order and let $g = \iota(\gbar)$ be the image of $\gbar$ in $\Aut L$. The standard lift of $g \in \Aut L$ to an automorphism $\hat{g}$ of the lattice-VOA $V_L$ does not have order doubling if and only $|B \cap \gbar^{\text{ord}(g)/2}(B)| \equiv 0 \bmod{4}$ for all $B \in C$.
\end{Theorem}

Since $g^{\text{ord}(g)/2}$ has order $2$, Corollary \ref{cor: orderdoubling} characterizes when $\hat{g}$ does not have order doubling in terms of the parity of $\langle \alpha, \gbar^{\text{ord}(g)/2}(\alpha) \rangle$ for all $\alpha \in L$. Thus the proof of Theorem~\ref{thm: orddoubcode} is an immediate consequence of the following lemma.

\begin{Lemma} \label{lem: order2autorddoubling} 
Suppose $\gbar \in \Aut C$ has order $2$ and let $g = \iota(\gbar)$ be the image of $\gbar$ in $\Aut L$. Then for all $B \in C$ we have $|B \cap g(B)| \equiv 0 \bmod{4}$ if and only if $\langle \alpha, g(\alpha) \rangle \in 2\Z$ for all $\alpha \in L$.
\end{Lemma}

\begin{proof}
For an arbitrary $\alpha \in L$ we may write $\alpha = \frac{1}{2} \alpha_B + \sum_{i=1}^N x_i \alpha_i$, for some $B \in C$. Let $\mathcal{I} \subset \{1, \dots, N\}$ denote the set of indices where $x_i \ne 0$. By equation \eqref{eqn:gaction}, the action of $g$ on $\alpha$ is given by $g(\alpha) = \frac{1}{2} \alpha_{\gbar(B)} + \sum_{i=1}^N x_i \alpha_{\gbar(i)}$. Thus we find 
\begin{align} \label{eqn: innerprod} \langle \alpha, g(\alpha) \rangle &= \begin{aligned}[t] & \frac{1}{4} \langle \alpha_B, \alpha_{\gbar(B)}\rangle + \frac{1}{2}\left( \sum_{i=1}^N \langle \alpha_B,  x_i \alpha_{\gbar(i)} \rangle + \sum_{i=1}^N \langle \alpha_{\gbar(B)}, x_i \alpha_i \rangle \right) \\&+\, \sum_{i=1}^N \sum_{j=1}^N \langle x_i \alpha_i , x_j \alpha_{\gbar(j)} \rangle 
\end{aligned}
\end{align}

We claim that the parity of $\langle \alpha, g(\alpha) \rangle$ is completely determined by whether $|B \cap \gbar(B)|$ is $0$ or $2$ modulo $4$. Since $\langle \alpha_i, \alpha_j \rangle = 2\delta_{i,j}$, the last term of \eqref{eqn: innerprod} is always even. If $\alpha \in L^{g}$ then the first two terms in \eqref{eqn: innerprod} simplify to
$ \frac{1}{4} \alpha_B^2 + \langle \alpha_B, \sum_{i=1}^N x_i \alpha_i \rangle \in 2\Z$
since $C$ is doubly even, therefore each codeword $B$ has weight $|B| \equiv 0 \bmod{4}$. 

So suppose that $\alpha \not \in L^g$.
Given any two subsets $X,Y \subset \{1, \dots, N\}$ we have $\gbar(X \cap Y) = \gbar(X) \cap \gbar(Y)$. Since $\gbar$ has order $2$ in $\Aut C$, we have that $B \cap \gbar(\mathcal{I}) = \gbar(\gbar(B) \cap \mathcal{I})$. Since $\gbar$ is a bijection, this implies $|B \cap \gbar(\mathcal{I})| =  |\gbar(B) \cap \mathcal{I}|$. 
This shows that the the second term of equation~\eqref{eqn: innerprod} is always even since the parity of the two sums is the same, determined by $|B \cap \gbar(\mathcal{I})|$. Thus the parity of $\langle \alpha, g(\alpha) \rangle$ is determined only by the first term of equation \eqref{eqn: innerprod}, i.e. by $\frac{1}{4} \langle \alpha_B, \alpha_{\gbar(B)} \rangle = \frac{1}{4} \cdot 2^{| B \cap \gbar(B)|}$. The statement of the lemma then follows.
\end{proof}
For codes $C$ of length $N \le 32$, when there exists an automorphism $\gbar$ with cycle type $2^{N/2}$ and fixed subcode satisfying the hypotheses of Theorem~\ref{thm: rA1^N/r} with $r=2$, then there are codewords $B \in C$ such $|B \cap g(B)| \equiv 2 \bmod{4}$. Hence by Theorem \ref{thm: orddoubcode}, the lift $\hat{g}$ must have order doubling. There appear to be many more cycle types which give rise to lattice automorphisms whose fixed sublattices have corresponding theta quotients that are not replicable functions. However, whenever there exists $\gbar \in \Aut C$ with cycle type $2^{N/2}$ and $C^\gbar$ satisfies the assumptions of Proposition~\ref{prop:2DN/2*}, 
each codeword $B \in C$ satisfies $|B \cap g(B)| \equiv 0 \bmod{4}$. Theorem \ref{thm: orddoubcode} implies that the standard lift of $g$ to $\hat{g} \in \Aut V_L$ does not have order doubling.

\section{Decomposition of VOA characters} \label{sec: decompVOAchars}
Since we have considered theta quotients coming from fixed sublattices under both cyclic and non-cyclic subgroups, we can also study characters of subVOAs fixed by noncyclic subgroups. This is more straightforward when none of the elements have order doubling, in which case we observe some decomposition of the fixed subVOA characters based on the group structure. It is of course possible to compute such characters for non-cyclic subgroups when order doubling occurs, however, the analogues of the following propositions are more complicated in those cases since the group structure as an subgroup of automorphism group of the lattice is not preserved (see, for example Section 4 of \cite{GK19}). Therefore, throughout this section, we assume all lifts of lattice automorphisms do not have order doubling.

\begin{Theorem*}[Theorem \ref{thm: charpq}] Let $p$ and $q$ be primes such that $q>p$ and $q\equiv 1 \pmod{p}$ and let $\Z_q \rtimes \Z_p $ be a subgroup of the automorphism group of an even positive definite lattice $L$. The characters of the fixed point subVOAs of the lattice-VOA $V_L$ satisfy the following relation
\[ p\Ch V^{\Z_q \rtimes \Z_p}= \Ch V^{\Z_q} +  p\Ch V^{\Z_p} - \Ch V_{L}. \]
\end{Theorem*}

\begin{proof}
Since $q >p$, this forces the existence of a unique $q$-Sylow subgroup. Since the group is not abelian and $q \equiv 1 \mod{p}$, the number of $p$-Sylow subgroups must be equal to $q$. Thus there are $(q-1)p$ elements of exact order $p$ and $q$ remaining elements of order dividing $q$.
Thus we have
\[ \Ch  V^{\Z_q \rtimes \Z_p}=\frac{1}{pq}\sum\limits_{g\in \Z_q \rtimes \Z_p} \tr (g| V_{L})=\frac{1}{pq}\left( 1\, \sum\limits_{g\in \Z_q} \tr (g| V_{L}) + q\sum\limits_{g\in \Z_p} \tr (g| V_{L}) -q \tr (e| V_{L})\right). \]

Now, $\Ch  V^{\Z_p}=\frac1p \sum_{g\in \Z_p}\tr (g| V_{L})$ and $\Ch  V^{\Z_q}=\frac1q \sum_{g\in \Z_q}\tr (g| V_{L})$, so we can rewrite the above as
\[ \Ch  V^{\Z_q \rtimes \Z_p}=\frac{1}{pq}\left( q\Ch  V^{\Z_q} + pq\Ch  V^{\Z_p} -q \Ch  V_{L}\right), \]
and thus 
\[ p\Ch  V^{\Z_q \rtimes \Z_p}= \Ch  V^{\Z_q} +  p\Ch  V^{\Z_p} - \Ch  V_{L}. \]
\end{proof}

In fact, these propositions hold more generally because $V$ need not be a lattice-VOA. Since we have assumed the lattice automorphism groups have no elements with order doubling, these propositions hold if one simply lets $G$ be a group of automorphisms of some strongly rational VOA $V$.

\begin{remark}
Note the formal similarities between Theorem~\ref{thm: charpq} and Proposition 3.4 of \cite{GK19},  in which they prove a similar statement about the characters of the $\Z_q \rtimes \Z_p$-orbifolds of $V_L$. Since the only holomorphic $C_2$-cofinite VOAs of CFT-type with central charge $c < 24$ are the lattice-VOAs $V_{E_8}$, $V_{D_{16}^+}$ and $V_{E_8^2}$ \cite{DM04}*{Theorems 1 and 2}, in the case of lattices $L$ of rank $N < 24$, the orbifolds of the lattice-VOA $V_L$ simply recover $V_L$ again, and so the statement of \cite{GK19}*{Proposition 3.4} becomes trivial in these cases. In contrast, the statement of Theorem~\ref{thm: charpq} regarding characters of fixed subVOAs can be applied to low rank lattices and we give an example below for the $E_8$ lattice to illustrate the result.
\end{remark}

\begin{Example}
The group generated by $h_1=(1, 2, 5, 3, 7, 6, 4)$ and $h_2=(2, 5, 7)(3, 4, 6)$ forms a subgroup  $H \cong \Z/7
\Z \rtimes \Z/3\Z$ of order $21$ in $\Aut{\mathcal{H}}$.
Let $H_1 = \langle h_1 \rangle \cong \Z/7\Z$ and $H_2 = \langle g_2 \rangle \cong \Z/3\Z$. Thus,
\[ 3\Ch V^{H} = \Ch V^{H_1}+ 3\Ch  V^{H_2}-\Ch  V_{E_8}.\]
This follows since $H$ has 14 elements of order 3, 6 elements of order 7, and one element of order 1. The elements of order 3 can be arranged into 7 subgroups of order 3 intersecting at the identity, similarly the remaining $6+1$ elements of order dividing 7 can be form a subgroup of order $7$. In particular, we find:
\begin{align*} \Ch V^{H} &= q^{-1/3}(1 + 22q + 242q^2 + 1762q^3 + 10460q^4 + 51078q^5 + 217266q^6 + \dots) \\
\Ch V^{H_1} &= q^{-1/3}(1 + 38q + 596q^2 + 4974q^3 + 
    30468q^4 + 151102q^5 + 647298q^6 + \dots) \\
 \Ch V^{H_2}&=q^{-1/3}(1 + 92q + 1418q^2 + 11688q^3 + 
    71346q^4 + 353212q^5+ 1511748q^6 + \dots) 
 \end{align*}
and finally, we have that $\Ch V^{H_1}+ 3\Ch  V^{H_2}-\Ch  V_{E_8}$
\begin{align*}  &= q^{-1/3}(3 + 66q + 726q^2+ 
    5286q^3 + 31380q^4 + 153234q^5+  651798q^6 + \dots)  \\ 
    & = 3\Ch V^{H}.\end{align*}
\end{Example}

\begin{Theorem}
 \label{thm: charp2q} Let $p$ and $q$   be primes such that $q>p$ and let $H$ be a non-abelian subgroup of order $p^2q$ of the automorphism group of an even positive definite lattice $L$. Then the characters of the fixed point subVOAs of the lattice-VOA $V_L$ satisfy the following relation
\begin{eqnarray} q\Ch V^{\Z_{p^2} \rtimes \Z_q} &= &\Ch V^{\Z_{p^2}} +  q\Ch V^{\Z_q} - \Ch V_{L}, \quad \text{ if } H \cong \Z_{p^2} \rtimes \Z_q \\ 
p^2 \Ch V^{\Z_q \rtimes \Z_{p^2}}&=& p^2 \Ch V^{\Z_{p^2}} +  \Ch V^{\Z_q} - \Ch V_{L},  \quad \text{ if } H \cong \Z_{q} \rtimes \Z_{p^2} \end{eqnarray}
\end{Theorem}

Much of the following proof is a standard exercise in group theory, however we record the details here for completeness.
\begin{proof} The number $n_q$ of $q$-Sylow subgroups of $H$ is either $1$ or $p^2$, since by assumption $p < q$ and therefore we cannot have $p \equiv 1 \bmod{q}$. If $n_q = p^2$ then $H$ contains $p^2$ subgroups of order $q$ that pairwise intersect in the identity element, thereby giving $(q-1)p^2$ elements of exact order $q$. This leaves $p^2q -(q-1)p^2 = p^2$ remaining elements, all of which must be contained in a single Sylow-$p$ subgroup which is necessarily normal in $H$. If instead $n_q = 1$, then the unique Sylow-$q$ subgroup is normal in $H$ and since $H$ is not abelian, we must have $n_p = q$ (and thus can only occur if $q \equiv 1 \bmod{p}$). In the former case, we have 
\[ \Ch V^H = \frac{1}{p^2 q}  \sum_{g \in H} \tr(g|V_L) = \frac{1}{p^2q} \left( \sum\limits_{g\in \Z_{p^2}} \tr (g | V_{L}) + p^2\sum\limits_{g\in \Z_q} \tr (g| V_{L}) -p^2 \tr (e| V_{L})\right) \]
Since $\Ch  V^{\Z_{p^2}}=\frac{1}{p^2} \sum\limits_{g\in \Z_p}\tr(g| V_{L})$ and $\Ch  V^{\Z_q}=\frac1q \sum\limits_{g\in \Z_q}\tr (g| V_{L})$, we can rewrite the above as
\[ \Ch  V^H=\frac{1}{p^2q}\left( p^2\Ch  V^{\Z_q} + p^2q\Ch  V^{\Z_p} -p^2 \Ch  V_{L}\right), \]
and thus 
\[ q\Ch  V^{H}= \Ch  V^{\Z_{p^2}} +  q\Ch  V^{\Z_q} - \Ch  V_{L}. \]
The proof for the latter case follows similarly.
\end{proof}

\begin{Question} Is there an analogous statement to Theorem \ref{thm: charpq} and Theorem \ref{thm: charp2q} that gives a decomposition of theta series of fixed sublattices under a group $G$ into theta series of fixed sublattices under subgroups of $G$? Alternatively, can a similar statement be made about the dimensions of the fixed sublattices under $G$ and its subgroups?
\end{Question}

\section{Questions for further study} \label{sec: questions}
We conclude by listing several additional questions for further study.

\begin{Question}
If one applies constructions other than Construction A to doubly even, self-dual linear codes (for example, Constructions B, C, or D), how do the replicable functions and graded characters compare to those that arise from Construction A? 
\end{Question}

\begin{Question} If one applies Construction A to a subcode fixed by an element $\bar{g} \in \Aut(C)$, the resulting lattice is distinct from the fixed sublattice $L^g$. Similarly, the lattice-VOA constructed from applying the lattice-VOA construction to (a potentially scaled version of) the fixed sublattice $L^g$ is not the same as the lattice-VOA $V_L^{\hat{g}}$. 
Can one quantify the relationship (if any) between the two lattices (resp. VOAs)?
\end{Question}

\begin{Question}
Consider the theta quotients associated to an automorphism of a lattice from Section~\ref{sec: repfunctionscodelats}, that is, the quotients of the theta functions of the fixed sublattices by the eta products based on the cycle type of the automorphism. Is there an analogue when one considers indefinite lattices that can be decomposed into positive definite and negative definite parts where the theta quotients are replaced by theta blocks of the form below?

\[ \dfrac{\theta_{L_1}(\tau) \theta_{L_2}(\tau)}{\eta(\tau)}\]

In a personal communication, J. Lagarias remarked that theta blocks are similar to the theta quotients considered by the authors and may be related to this work (see \cite{GPSZ19} for the definition of theta block in general). The authors speculate the aforementioned connection with indefinite lattices, but there may be a different connection altogether. 
\end{Question}

\begin{Question}
Can one characterize how the structure of a subgroup of $\Aut(C)$ changes when lifted to a subgroup of VOA automorphisms in cases where elements of the subgroup have order doubling? Is there anything to be learned from restricting our attention only to lattice automorphisms that come from code automorphisms?
\end{Question}

\begin{Question}
Are there characteristics of the fixed subVOAs under lifted code automorphisms that can be deduced from the fixed subcode under that automorphism? For example, the conformal weights of twisted modules of lattice-VOAs by VOA automorphisms depends on quantities related to the lattice and the cycle type of the automorphism. Can this or other values be computed using properties of the fixed subcodes?
\end{Question}

\bibliographystyle{alpha}
\bibliography{ref}

\end{document}